\documentclass[11pt]{article}

\newcommand{\ScalingMachineDescription}{%
  workstation with an AMD Ryzen AI 9 HX 370 processor (12 cores, 24 threads)
  and 124\,GiB of RAM, running TUXEDO OS%
}

\usepackage{amsmath,amssymb,amsthm,mathtools}
\usepackage{mathrsfs}
\usepackage[a4paper,margin=1in]{geometry}
\usepackage{bm}
\usepackage{multirow}   
\usepackage{booktabs}   
\usepackage{siunitx}  
\usepackage{xcolor}
\usepackage{hyperref}

\usepackage{graphicx}
\usepackage{subcaption}
\usepackage{float}

\newtheorem{definition}{Definition}
\newtheorem{theorem}{Theorem}
\newtheorem{lemma}{Lemma}
\newtheorem{remark}[theorem]{Remark}
\newtheorem{corollary}{Corollary}

\newcommand{\Th}{\mathcal{T}_h}
\newcommand{\Eh}{\mathcal{E}_h}
\newcommand{\Eho}{\mathcal{E}_h^o}
\newcommand{\Ehb}{\mathcal{E}_h^{\partial}}
\newcommand{\Ehj}{\mathcal{E}_h^{J}}

\newcommand{\wh}{\mathbf{w}_h}

\newcommand{\who}{\widehat{\mathbf{w}}_h}

\title{An HDG-type discretization of 1D embedded blood flow in complex networks with Riemann-coupled junctions}
\author{}
\date{}

\author{
Raksha Devi$^{1}$\thanks{Corresponding author: \texttt{raksha.devi@dm.unipi.it}.} \and
Luca Heltai$^{1}$ \and
Lucas O. M\"uller$^{2}$
}

\date{}

\begin{document}
\maketitle
\begingroup
\renewcommand{\thefootnote}{}
\footnotetext{\textit{E-mail addresses:}
\texttt{luca.heltai@unipi.it} (L.~Heltai),
\texttt{lucas.muller@unitn.it} (L.O.~M\"uller).}
\endgroup

\begin{center}
$^{1}$ Department of Mathematics, University of Pisa, Italy\\
$^{2}$ Laboratory of Mathematics for Biology and Medicine, Department of Mathematics, University of Trento, Trento, Italy
\end{center}

\begin{abstract}
We present a hybridizable discontinuous Galerkin (HDG) method for one-dimensional blood flow in arterial networks with Windkessel boundary conditions. The scheme uses a Harten--Lax--van Leer (HLL) numerical flux and implicit time integration. After spatial discretization, the complete network---including cell states, hybrid traces, junction conditions, and dynamic $\mathrm{RCR}$ terminal pressures---is written as a residual index-one differential--algebraic equation (DAE) and integrated with variable-step, variable-order backward differentiation formulas. The implicit formulation removes the CFL stability restriction of explicit schemes and allows variable-order, variable-step solvers. The resulting single global residual and Jacobian for the network and its terminal dynamics also provide a natural interface for monolithic coupling with external nonlinear models. We establish local solvability and a semi-discrete energy balance, and show numerically optimal convergence rates for both solution variables in the $L^2$ norm and the broken $H^1$ seminorm. The solver is validated against established benchmarks, including an aortic bifurcation and a $37$-artery network, and is further applied to a cerebral arterial network. Mass conservation is observed in a synthetically generated anatomical liver network of $8\,\mathrm{k}$ vessels, and strong-scaling studies are performed on synthetic networks of $50\,\mathrm{k}$ and $100\,\mathrm{k}$ vessels.
\end{abstract}
\section{Introduction}
One-dimensional ($1\mathrm{D}$) models of the cardiovascular system provide a computationally efficient framework for analyzing pulse-wave propagation and hemodynamic variables across complex arterial networks \cite{formaggia2003one,sherwin2003computational,blanco2025one,kim2025one,boileau2015benchmark}. By simplifying the three-dimensional Navier--Stokes equations through cross-sectional averaging, these reduced-order models retain the dominant network-scale dynamics at a small fraction of the cost of fully three-dimensional simulations. In the literature, each vessel is commonly represented as an interval with a single axial coordinate. However, anatomically realistic arteries are curved tubes with a definite orientation in space, and effects such as gravity depend on that orientation. In existing $1\mathrm{D}$ models, gravity is typically introduced through a source term involving the component of gravitational acceleration along the vessel axis \cite{MR5106706,shramko2023gravity}.

In view of possible coupling with three-dimensional elasticity models~\cite{MR1873203,heltai2021multiscale,belponer2023reduced}, we incorporate the orientation into the generic formulation while retaining a genuinely one-dimensional network discretization. Each vessel is parameterized by arc length $s$ along a centerline curve embedded in $\mathbb{R}^3$, with unit tangent $\bm{b}(\bm{x})$. The spatial flux tensor is $\mathbf F(\mathbf w,s)\otimes\bm b$, and its row-wise intrinsic divergence satisfies $\operatorname{div}_\Gamma\bigl(\mathbf F(\mathbf w,s)\otimes\bm b\bigr)=\partial_s\mathbf F(\mathbf w,s)$. Thus, the embedding is purely \emph{geometric}: the model describes one-dimensional flow along the centerlines, with $\bm b(\bm x)$ specifying the local vessel orientation. No ambient divergence or off-centerline extension of the tangent field is required; the full formulation is given in Section~\ref{sec:model-formulation}. When gravity is included, its projection onto the vessel enters the momentum source through $\bm g\!\cdot\!\bm b(\bm x)$. A spatially varying $\bm b(\bm x)$ therefore allows this known gravity-dependent source term to follow the vessel geometry accurately.

 A network model is closed by two types of coupling. At each junction one imposes conservation of mass and continuity of total pressure across the incident vessels. At each terminal we attach a lumped $0\mathrm{D}$ Windkessel model whose pressure satisfies an ordinary differential equation (ODE). Existing network formulations include characteristic, finite-volume, and discontinuous Galerkin discretizations, as well as hybrid $1\mathrm{D}$--$0\mathrm{D}$ models that couple different vessel descriptions through conservation conditions \cite{formaggia2003one,ghitti2023hybrid}, and terminal coupling is standard in these works. Implicit arterial-network solvers also predate the present formulation: Carson and van Loon~\cite{carson2017implicit} developed an enhanced trapezoidal rule method with a global pressure solve and Lagrange multipliers enforcing conservation of mass and total-pressure continuity at junctions. The present contribution combines a specific hybridizable discontinuous Galerkin (HDG) formulation with HLL-based stabilization, the coupling of multi-way junctions and dynamic Windkessel terminals through hybrid trace unknowns, and implicit time integration with adaptive step size and variable order. Terminal pressures enter the global unknowns alongside the traces, so that vessel, junction and terminal equations form a single nonlinear index-one differential–algebraic system (DAE) in residual form. The system is integrated implicitly by a variable-order, variable-step backward differentiation formula ($\mathrm{BDF}$). Each time step is solved by Newton’s method on the full nonlinear system, with no linearization or operator splitting, using one global Jacobian that couples vessels, junctions and terminals.

This monolithic structure is particularly well suited to simulating networks in which entire subnetworks, including those representing whole organs, are replaced by lumped $0\mathrm{D}$ models, even when their governing ODEs are complex and highly stiff. Incorporating such a model requires augmenting the global unknowns with its state variables and appending its governing equations and coupling conditions to the same residual and Jacobian, just as for hybrid traces and terminal pressures. The $\mathrm{RCR}$ Windkessel model used here provides a simple example of this integration; the same assembly strategy extends naturally to more elaborate organ models, while the implicit time integrator accommodates their stiffness.

Implementing this strategy requires a global set of unknowns that couples hybrid traces, junctions, and terminals within the same nonlinear system. The hybridizable discontinuous Galerkin (HDG) method provides such a set. Discontinuous Galerkin (DG) methods are widely used for hyperbolic conservation laws because they handle shocks and complex geometries and are well suited to parallel computation. With implicit time integration, however, a standard DG formulation leads to a globally coupled system involving all element degrees of freedom. With explicit time integration, by contrast, each element is updated from known neighbour states and no global solve is needed, at the price of a CFL-type step restriction.
HDG methods introduce additional ``trace'' unknowns defined only on the mesh skeleton (the traces of the unknowns on the mesh faces). Once those trace values are known, each element can be solved independently using a local element solver \cite{MR3345205}. The junction and terminal conditions involve these skeleton unknowns and can be assembled with the cell equations in a single Newton system. Although element elimination is possible when the local Newton blocks are invertible, the current solver implementation solves the full cell-trace-capacitor system.

HDG methods have been extensively studied over the last two decades, and several HDG formulations incorporating approximate
Riemann solvers such as the Harten--Lax--van Leer (HLL)
flux have been proposed for hyperbolic systems,
particularly in the context of compressible fluid
dynamics \cite{peraire2010hybridizable,vila2021hybridisable}.
For $1\mathrm{D}$ blood flow, DG methods have been used for reduced models \cite{puelz2017comparison,masri2022discontinuous} and for coupled flow and solute transport on networks \cite{masri2022discontinuous}. The novelty of the present work lies in combining the HDG--HLL network formulation and its Riemann-coupled junction and terminal treatment with a residual DAE formulation integrated by variable-order, variable-step $\mathrm{BDF}$ methods. The analytical difficulty is concentrated at the junctions: single-valued traces are sufficient for ordinary interfaces, whereas multiple connected vessels require the simultaneous imposition of mass conservation, total-pressure continuity, and characteristic compatibility across all incident vessels. The principal characteristics of the implemented solver are as follows.
\begin{enumerate}
    \item An HDG--HLL discretization of the $1\mathrm{D}$ network that incorporates junctions and $\mathrm{RCR}$ terminals into one global cell–trace–capacitor system.
    \item An implicit index-one DAE formulation in which cell and terminal-capacitor variables are differential and hybrid traces are algebraic, together with local solvability and a semi-discrete energy estimate.
    \item Observed optimal spatial rates on a single-vessel manufactured solution in the $L^2$ norm and broken $H^1$ seminorm for area and velocity.
    \item Validation against aortic bifurcation and $37$-artery benchmarks, application to a cerebral arterial network, a mass-conservation study in an anatomical liver network of $8\,\mathrm{k}$ vessels, and strong-scaling studies in synthetic networks of $50\,\mathrm{k}$ and $100\,\mathrm{k}$ vessels.
\end{enumerate}
The solver structure provides an interface for full monolithic coupling with external solvers, regardless of their nature. For example, an external model with unknowns $z$ can be appended schematically as
\begin{equation}
\mathcal R_{\mathrm{flow}}(t,y,\dot y;z)=0,\qquad
\mathcal R_{\mathrm{ext}}(t,z,\dot z;y)=0,
\qquad \mathcal C(y,z)=0,
\label{eq:future_external_coupling}
\end{equation}
so that the flow, external-model, and interface equations are assembled into the same Newton system. The DAE formulation thus provides a monolithic coupling interface.

The remainder of the paper is organized as follows. Section~\ref{sec:model-formulation} presents the governing $1\mathrm{D}$ system, tube law, and characteristic structure. Section~\ref{sec:hdg-formulation} develops the HDG--HLL discretization. Section~\ref{sec:transmission} discusses the unified treatment of transmission, boundary, and junction conditions. In Section~\ref{sec:analysis} we establish local solvability and derive the energy balance of the semi-discrete scheme. Section~\ref{sec:time_integration} discusses the implicit time-integration strategy and the associated nonlinear solvers. Finally, Section~\ref{sec:numerical_results} presents convergence results, benchmark validation, a cerebral network application, a mass-conservation study, and strong-scaling results for large vascular networks.

\section{Governing equations}\label{sec:model-formulation}
Each vessel $e$ has a fixed, oriented $C^2$ centerline
$\gamma_e:[0,L_e]\to\mathbb{R}^3$, parameterized by arc length $s$, with unit
tangent $\bm b_e(\gamma_e(s))=\gamma_e'(s)$. Positive velocity is directed toward
increasing $s$; vessel subscripts are suppressed below. For a scalar $v$ we write
$\nabla_\Gamma v=\bm b\,\partial_s v$, and for $\mathbf v\in\mathbb{R}^2$,
$\nabla_\Gamma\mathbf v=\partial_s\mathbf v\otimes\bm b\in\mathbb{R}^{2\times3}$.
We define the spatial flux tensor
\begin{equation}\label{eq:flux_tensor}
  \mathbb{F}(\mathbf w,s):=\mathbf F(\mathbf w,s)\otimes\bm b\in\mathbb{R}^{2\times3},
\end{equation}
whose $i$-th row is $F_i\bm b$, and take its intrinsic divergence row by row,
\begin{equation}\label{eq:intrinsic_div}
  \big(\operatorname{div}_\Gamma\mathbb{F}\big)_i
  =\operatorname{div}_\Gamma(F_i\bm b)
  =\bm b\cdot\partial_s(F_i\bm b)
  =\partial_s F_i ,
\end{equation}
where we use the fact that $|\bm b|=1$ implies $\bm b\cdot\partial_s\bm b=0$. The vessel-wise system is
\begin{equation}\label{equation_system}
  \partial_t\mathbf w+\operatorname{div}_\Gamma\big(\mathbf F(\mathbf w,s)\otimes\bm b\big)
  =\mathbf S(\mathbf w,s),
  \qquad
  \mathbf w=\begin{pmatrix}A\\ U\end{pmatrix},
\end{equation}
where the physical flux and source are
\begin{equation}\label{eq:flux_source}
  \mathbf F(\mathbf w,s)=
  \begin{pmatrix}
    AU\\[2pt]
    \dfrac{U^2}{2}+\dfrac{1}{\rho}P(A,s)
  \end{pmatrix},
  \qquad
  \mathbf S(\mathbf w,s)=
  \begin{pmatrix}
    0\\[2pt]
    -\dfrac{fU}{A}
  \end{pmatrix}.
\end{equation}
Here $A(s,t)$ is the cross-sectional area, $U(s,t)$ the average axial velocity,
$\rho$ the blood density, and $f=2(\zeta+2)\mu\pi/\rho$ the friction coefficient, with
$\mu$ the dynamic viscosity and $\zeta$ the velocity-profile parameter. When a
gravitational body force $\bm g$ is included, the momentum source gains the known
term $\bm g\cdot\bm b$, which follows the local vessel orientation.

At an element endpoint the outward tangent is $\bm n=\sigma\bm b$, with
$\sigma=-1$ at $s=s_{\min}$ (conventionally called ``the left
part'' of the vessel) and $\sigma=+1$ at $s=s_{\max}$ (conventionally
called ``the right part'' of the vessel), so that $\mathbb{F}\bm
n=\sigma\mathbf F$. With $(\mathbb{A},\mathbb{B})_K:=\int_K
\mathbb{A}:\mathbb{B}\,\mathrm ds$, integration by parts along the centerline
gives
\begin{equation}\label{eq:ibp}
  \big(\operatorname{div}_\Gamma\mathbb{F},\mathbf v\big)_K
  =-\big(\mathbb{F},\nabla_\Gamma\mathbf v\big)_K
  +\big\langle\mathbb{F}\bm n,\mathbf v\big\rangle_{\partial K},
  \qquad 
  \mathbb{F}:\nabla_\Gamma\mathbf v=\mathbf F\cdot\partial_s\mathbf v .
\end{equation}
The curvature term does not appear because every row of $\mathbb{F}$ is tangent
to the centerline. All integrals use arc length, and no ambient divergence or
extension of $\bm b$ is required. Junctions carry one tangent per incident
vessel. The simulations use straight vessels, but the embedded notation also
permits curved centerlines, without modeling curvature-induced secondary flow or
junction losses.

The arterial wall mechanics are governed by the tube law~\cite{matthys2007pulse}
\begin{equation}\label{eq:tube_law}
  P(A,s)-P_0=P_d+\frac{\beta(s)}{A_d(s)}\big(A^m-A_d(s)^m\big),
  \qquad
  \beta=\tfrac43\sqrt{\pi}\,E\,h_w(s),
  \qquad m=\tfrac12,
\end{equation}
where $E$ is the Young's modulus, $h_w$ the wall thickness, and $A_d(s)$ the local
diastolic area. For a tapered vessel, $A_d(s)=\pi r_d(s)^2$ with $r_d$ interpolated
linearly between the inlet radius $r_{\mathrm{in}}$ and the outlet radius
$r_{\mathrm{out}}$~\cite{alastruey2012arterial,boileau2015benchmark},
\begin{equation}\label{eq:taper}
  r_d(s)=r_{\mathrm{in}}+\frac{s-s_{\min}}{s_{\max}-s_{\min}}
  \big(r_{\mathrm{out}}-r_{\mathrm{in}}\big),
\end{equation}
where $s_{\min}$ and $s_{\max}$ are the arc-length bounds of the vessel. For a
uniform vessel, $r_{\mathrm{in}}=r_{\mathrm{out}}=r_d$ and $A_d=\pi r_d^2$ is constant.

The wave speed is
\begin{equation}\label{eq:wave_speed}
  c(A,s)=\sqrt{\frac{A}{\rho}\,\partial_A P}
  =\sqrt{\frac{\beta}{2\rho A_d}}\,A^{1/4},
\end{equation}
and, at fixed $s$, the flux Jacobian is
\begin{equation}\label{eq:jacobian}
  \frac{\partial\mathbf F}{\partial\mathbf w}=
  \begin{pmatrix}
    U & A\\
    c^2/A & U
  \end{pmatrix},
\end{equation}
with eigenvalues $\lambda_{1,2}=U\pm c$. Under physiological conditions
$|U|<c$, the system~\eqref{equation_system} is strictly hyperbolic. For locally constant
coefficients it has the Riemann invariants $W_{1,2}=U\pm4(c-c_0)$ with
$c_0=c(A_d,s)$; for spatially varying coefficients these are evaluated locally at
the endpoint where they are used.
\section{Hybridizable Discontinuous Galerkin Framework}\label{sec:hdg-formulation}
Let $\Th$ be the collection of non-overlapping line elements $K$, and let $\Eh$ be the set of all boundary endpoints (nodes), whose elements we call faces and denote by
$F$. We split $\Eh$ into three mutually disjoint subsets:
\[
  \Eh = \Eho \cup \Ehb \cup \Ehj,
\]
representing standard interior nodes ($\Eho$), exterior boundary terminals ($\Ehb$), and branching junctions ($\Ehj$). Following the standard HDG construction~\cite{MR3345205,peraire2010hybridizable}, we introduce a local element space $\mathbf{W}_h^k$ and the trace space $\mathbf{M}_h^k$
\begin{align*}
  \mathbf{W}_h^k &= \left\{ \mathbf{v} \in (L^2(\Th))^2 : \mathbf{v}|_K \in (\mathcal{P}^k(K))^2, \, \forall K \in \Th \right\}, \\
  \mathbf{M}_h^k &= \left\{ \bm{\mu} \in (L^2(\Eh))^2 : \bm{\mu}|_F \in (\mathcal{P}^k(F))^2, \, \forall F \in \Eh \right\}.
\end{align*}
We seek $(\mathbf w_h,\who)\in\mathbf W_h^k\times\mathbf M_h^k$ with
$\who=(\widehat A_h,\widehat U_h)^\top$. 

Since the mesh is one-dimensional, every face $F \in \Eh $ is a single point with arc-length coordinate $s_F$, and $\mathcal{P}^k(F)^2 = \mathbb R^2$. The trace space $\mathbf{M}_h^k$ is therefore independent of the polynomial degree $k$. On each face the trace is a single pair of values $\who (F)=(\widehat A_h,\widehat U_h)(F) \in \mathbb{R}^2$. We discuss the trace values at interior faces, boundary faces, and junctions in detail in Section~\ref{sec:transmission}.

Testing \eqref{equation_system} on each element, integrating by parts, and replacing the boundary flux by a numerical flux gives the local problem
\begin{equation}\label{eq:hdg_local}
  \left(\partial_t\mathbf w_h,\mathbf v\right)_K
  -\left(\mathbb{F}(\mathbf w_h,s),\nabla_\Gamma\mathbf v\right)_K
  +\left\langle\widehat{\mathbb{F}}_h\bm n,\mathbf v\right\rangle_{\partial K}
  =\left(\mathbf S(\mathbf w_h,s),\mathbf v\right)_K,
  \quad\forall\,\mathbf v\in(\mathcal P^k(K))^2,
\end{equation}
where $\bm n = \sigma \bm b$ is the outward tangent at an endpoint of $K$, with $\sigma=-1$ on the left and $\sigma=+1$ on the right, and
$\langle\cdot,\cdot\rangle_{\partial K}$ denotes the sum of point evaluations
at the two endpoints. By~\eqref{eq:ibp}, the volume term equals
$\left(\mathbf F(\mathbf w_h,s),\partial_s\mathbf v\right)_K$. The term
$\widehat{\mathbb{F}}_h\bm n$ is the two-component numerical approximation of
$\mathbb{F}\bm n=\sigma\mathbf F$, and we write it as
$\widehat{\mathbf F}_h\cdot\bm n$ in what follows. The interior state on each
element is coupled to its traces, while the traces are fixed globally by the
transmission, boundary, and junction conditions of
Section~\ref{sec:transmission}.
\begin{figure}[htbp]
  \centering
  \includegraphics[width=\textwidth]{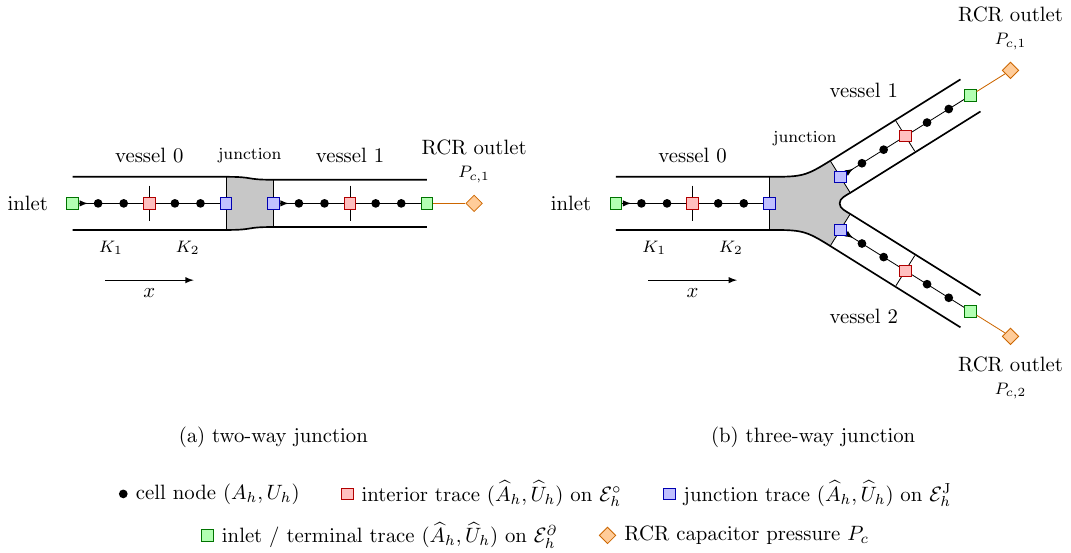}
  \caption{HDG DoF at $2$-way junctions and a vessel bifurcation with polynomial degree ($p=1$).}
  \label{fig:bifurcation_dofs}
\end{figure}
For the configuration in Figure~\ref{fig:bifurcation_dofs}(b), we have $3$ vessels, each with $N_{\mathrm{el}}=2$ elements and $p=1$. Each element carries $2(p+1)=4$ unknowns. That gives $3\,N_{\mathrm{el}}\cdot2(p+1)= 24$ cell unknowns (black dots). There are $3$ trace pairs at the interior nodes, $3$ at the inlet and outlets, and $3$ at the junction. These give $18$ trace unknowns. The two $\mathrm{RCR}$ outlets add $2$ capacitor pressures, so the system has $24+18+2=44$ DoFs in total. The DoFs for the $2$-way junction in Figure~\ref{fig:bifurcation_dofs}(a) are counted in the same way.

\subsection{Numerical trace flux}\label{sec:num-flux}
The numerical trace flux is a two-state function of the interior (self) state
$\wh$ and the trace state. On any face, it receives the interior element state
$\mathbf w_h=(A_h,U_h)^\top$ (the \emph{self} state, evaluated on the face
from inside $K$) and the single-valued trace
$\widehat{\mathbf w}_h=(\widehat A_h,\widehat U_h)^\top$ (the
\emph{exterior} state living on the face), and
returns a single normal flux,
\begin{equation}
  \widehat{\mathbf F}_h\cdot\bm n
  \;=\;
  \mathcal F_{\mathrm{num}}\bigl(\,b\!\cdot\!\bm n,\;\mathbf w_h,\;
  \widehat{\mathbf w}_h\,\bigr).
  \label{eq:numflux_two_state}
\end{equation}
We use $\mathcal F_{\mathrm{num}}$ to denote a standard numerical flux,
such as Lax--Friedrichs, Rusanov, Roe, or HLL \cite{toro2013riemann}. In HDG form, these fluxes can
be written as
\begin{equation}
  \widehat{\mathbf F}_h\cdot\bm n
  \;=\;
  \sigma \mathbf F(\widehat{\mathbf w}_h, s_F)
  \;+\;
 \bm\Sigma(\mathbf w_h,\widehat{\mathbf w}_h)\,
  \bigl(\mathbf w_h-\widehat{\mathbf w}_h\bigr)
  \label{eq:hdg_template}
\end{equation}
The first term is the consistent part, with the physical flux evaluated at the trace. The second is the stabilisation (dissipation) part. The matrix
$\bm\Sigma$ penalizes the interior--trace mismatch, controls the amount of
numerical dissipation, and contributes to the stability of the hybridized
formulation.

The simplest closure is the scalar trace anchored penalty
$\bm\Sigma_{\mathrm{s}}=\alpha_{\mathrm{s}}\,\mathbf I$ with
$$\alpha_{\mathrm{s}}=\max\bigl(|U_h\pm c_h|,\,|\widehat U_h\pm\widehat
c_h|\bigr),$$ which applies the same dissipation to both characteristic
directions. A second possibility is the two-state local Lax--Friedrichs (Rusanov) flux,
which in the form~\eqref{eq:hdg_template} corresponds to
$\bm\Sigma=\tfrac12(\alpha\mathbf I+\sigma\bar{\mathbf J})$, with
$\bar{\mathbf J}$ the path-averaged flux Jacobian, which is defined as
\begin{equation}
  \bar{\mathbf J}=\int_0^1\partial_{\mathbf w}\mathbf F\bigl(\widehat{\mathbf w}_h
  +\theta(\mathbf w_h-\widehat{\mathbf w}_h),s_F\bigr)\,\mathrm d\theta .
  \label{eq:path_jacobian}
\end{equation}

The HLL flux instead uses separate estimates of the left- and right-running wave speeds. It therefore distinguishes the two characteristic directions and provides wave-dependent dissipation.

\textit{From the HDG stabilisation to the HLL flux.}
In the HDG context, the self (interior) state is the
left state and the trace is the right state,
\begin{equation}
  \mathbf w_L \equiv \mathbf w_h = (A_h,U_h)^\top,
  \qquad
  \mathbf w_R \equiv \widehat{\mathbf w}_h = (\widehat A_h,\widehat U_h)^\top .
\end{equation}
The orientation-projected physical fluxes are
\begin{equation}
  \mathbf F_\star\cdot\bm n =
\sigma \left(
  A_\star U_\star\,,  \,
  \tfrac12 U_\star^2 + P(A_\star, s_F)/\rho
  \right)^{\top}
  \qquad \star\in\{L,R\}.
  \label{eq:proj_flux}
\end{equation}
The normal flux Jacobian $\sigma\,\partial_{\mathbf w}\mathbf F$ has the
characteristic speeds $\sigma U\pm c$. The signal speeds are estimated from
the averaged states,
\begin{equation}
  \bar U=\tfrac12(U_L+U_R),
  \quad
  \bar c=\tfrac12\bigl(c(A_L)+c(A_R)\bigr),
  \qquad
  s_L=\sigma \bar U-\bar c,
  \quad
  s_R= \sigma \bar U+\bar c .
  \label{eq:wave_speeds}
\end{equation}
The HLL numerical flux is the three-branch expression \cite[Chapter~10]{toro2013riemann}
\begin{equation} \label{eq:hll_flux}
  \widehat{\mathbf F}_h\cdot\bm n
  =
  \begin{cases}
  \mathbf F_R\cdot\bm n, & s_R\le 0,\\[6pt]
  \dfrac{s_R\,\mathbf F_L\cdot\bm n - s_L\,\mathbf F_R\cdot\bm n
         + s_R s_L\,(\mathbf w_R-\mathbf w_L)}{s_R-s_L}, 
  & s_L<0<s_R, \\[6pt]
  \mathbf F_L\cdot\bm n, & s_L\ge 0.
  \end{cases}
\end{equation}
In the subcritical physiological regime $|U|<c$, the central branch is active, and the
dissipative jump $s_Rs_L(\mathbf w_R-\mathbf w_L)$ acts on the
\emph{raw} states.

\textit{Equivalence with the HDG normal flux.}
The central branch of~\eqref{eq:hll_flux} can
be decomposed into a consistent part and a \emph{scalar} dissipation acting on
the interior--trace jump. Splitting the
numerator gives
\begin{equation}
  \widehat{\mathbf F}_h\cdot\bm n
  =\underbrace{\frac{\sigma\bigl(s_R\,\mathbf F(\mathbf w_h,s_F)
            -s_L\,\mathbf F(\who,s_F)\bigr)}{s_R-s_L}}
            _{\text{consistent (wave-weighted) flux}}
  +\underbrace{\delta_{\mathrm{HLL}}\,
  \bigl(\mathbf w_h-\who\bigr)}_{\text{dissipation}},
  \qquad
  \delta_{\mathrm{HLL}}=-\frac{s_L\,s_R}{s_R-s_L}>0,
  \label{eq:hll_consistent}
\end{equation}
where positivity of $\delta_{\mathrm{HLL}}$ follows from $s_L<0<s_R$. Rewriting~\eqref{eq:hll_consistent} in the form~\eqref{eq:hdg_template} gives
the matrix-valued coefficient
\begin{equation}
  \bm\Sigma_{{\mathrm{HLL}}}=\delta_{\mathrm{HLL}}\mathbf I
  +\frac{\sigma\,s_R}{s_R-s_L} \bar{\mathbf J},
  \label{eq:hll_trace_anchored}
\end{equation}
which also contains the averaged flux Jacobian of the two states.

\begin{remark}[HLL $\to$ Rusanov]
With a single symmetric speed $s_R=-s_L=\alpha_{\mathrm s}$, one has
$\delta_{\mathrm{HLL}}=\alpha_{\mathrm s}/2$ and $s_R/(s_R-s_L)=\tfrac12$, so
$\bm\Sigma_{\mathrm{HLL}}$ in~\eqref{eq:hll_trace_anchored} reduces to the
Rusanov coefficient $\tfrac12(\alpha_{\mathrm s}\mathbf I+\sigma\bar{\mathbf J})$.
HLL is therefore the two-wave generalisation of Rusanov, weighting the
forward ($U+c$) and backward ($U-c$) characteristics separately.
\end{remark}

\section{Transmission, boundary and junction conditions}\label{sec:transmission}
The traces are fixed globally by transmission, boundary and junction conditions. At a junction node the trace space carries one trace pair per incident vessel, and at every other node it carries a single pair. We seek
$\widehat{\mathbf w}_h\in\mathbf M_h^k$ such that
\begin{equation}
  \sum_{K\in\Th}
  \bigl\langle\widehat{\mathbf F}_h\cdot\bm n,\bm\mu\bigr\rangle
    _{\partial K\setminus(\Ehb\cup\Ehj)}
  +\sum_{F\in\Ehb}\bigl\langle\widehat{\mathbf B}_h,\bm\mu\bigr\rangle_F
  +\sum_{J\in\Ehj}\bigl\langle\widehat{\mathbf J}_h,\bm\mu\bigr\rangle_J
  =0,
  \qquad\forall\bm\mu\in\mathbf M_h^k .
  \label{eq:global_hdg}
\end{equation}
The interior terms use the same HLL flux \eqref{eq:hll_flux} as the local problem \eqref{eq:hdg_local}, so the element state and the traces are coupled consistently. The operators
$\widehat{\mathbf B}_h$ and $\widehat{\mathbf J}_h$
encode the boundary and junction physics defined next. All tube-law
quantities at a face $F$ are evaluated at its arc-length position $s_F$ in the
incident vessel: $P(\widehat A)=P(\widehat A,s_F)$,
$\widehat c=c(\widehat A,s_F)$, and $c_0=c(A_d(s_F),s_F)$.

\subsection{Interior faces}\label{sec:interior}
Consider a face shared by $K_L$ and $K_R$ of the same vessel, so that $\sigma_L =+1$ and $\sigma_R =-1$.
The residual is the sum of the two one-sided HLL fluxes, each taking that side's interior as the self state and the shared trace as the exterior state:
\begin{equation}
  \mathcal F_{\mathrm{int}}
  =
  \widehat{\mathbf F}_h\bigl(\sigma_L,\,\mathbf w_L,\,
  \widehat{\mathbf w}_h\bigr)
  +
  \widehat{\mathbf F}_h\bigl(\sigma_R,\,\mathbf w_R,\,
  \widehat{\mathbf w}_h\bigr)
  =\mathbf 0 .
  \label{eq:interior_hll}
\end{equation}
These are two scalar equations (area and momentum rows) fixing the two trace
unknowns $(\widehat A_h,\widehat U_h)$. 
\subsection{Physical boundaries: \texorpdfstring{$\widehat{\mathbf{B}}_h$}{B-hat}}\label{sec:boundaries}
At a subcritical boundary face exactly one characteristic leaves the vessel,
so exactly one physical condition can be imposed. The operator $\widehat{\mathbf B}_h$
combines it with one outgoing Riemann-invariant compatibility condition,
using $W_{1,2}=U+4\sigma(c-c_0)$, where $\sigma$ is the outward
orientation of the face. Inlets are placed at $s=0$ ($\sigma=-1$) and outlets
at $s=L_e$ ($\sigma=+1$) of their vessels.

\textit{Inflow.} A prescribed flow rate $Q_{\mathrm{in}}(t)$ is combined with the outgoing backward invariant:
\begin{equation}\label{eq:bc_inflow}
  \widehat{\mathbf B}_h^{\mathrm{in}}
  =
  \begin{pmatrix}
  \widehat A_h\,\widehat U_h - Q_{\mathrm{in}}(t)\\[6pt]
  \bigl[\widehat U_h-4(\widehat c_h-c_0)\bigr]
  -\bigl[U_h-4(c_h-c_0)\bigr]
  \end{pmatrix}
  =\mathbf 0 .
\end{equation}

\textit{Windkessel/RCR outlet.} Pressure continuity with the three-element Windkessel is combined with the outgoing forward invariant:
\begin{equation}\label{eq:bc_rcr}
  \widehat{\mathbf B}_h^{\mathrm{rcr}}
  =
  \begin{pmatrix}
  P(\widehat A_h) - \bigl(R_1\,\widehat A_h\widehat U_h + P_c\bigr)\\[6pt]
  \bigl[\widehat U_h+4(\widehat c_h-c_0)\bigr]
  -\bigl[U_h+4(c_h-c_0)\bigr]
  \end{pmatrix}
  =\mathbf 0 ,
\end{equation}
where the capacitor pressure $P_c$ is a global unknown of the DAE. For a
capacitance $C>0$, it evolves according to
\[
C\dot P_c = \widehat A_h\widehat U_h - \frac{P_c-P_{\mathrm{out}}}{R_2},
\]
which we write as the residual
\begin{equation}
R_{pc} = C \dot P_c  -
\widehat A_h\widehat U_h
+
\frac{P_c-P_{\mathrm{out}}}{R_2}.
\label{eq:rcr_capacitor}
\end{equation}
Because $P_c$ couples to the trace through \eqref{eq:bc_rcr}-\eqref{eq:rcr_capacitor},
the terminal contributes to the same Jacobian as the vessel variables. The
corresponding blocks $D$, $C$, and $E$ are defined in
Section~\ref{sec:time_integration}.

\textit{Resistive outlet.} When $C=0$, the RCR model reduces to a single resistance. The capacitor equation becomes algebraic,
$P_c=P_{\mathrm{out}}+R_2\,\widehat A_h\widehat U_h$, and the first row
of~\eqref{eq:bc_rcr} reduces to
\begin{equation}\label{eq:bc_resistive}
  P(\widehat A_h)-P_{\mathrm{out}}-R\,\widehat A_h\widehat U_h=0,
  \qquad R=R_1+R_2,
\end{equation}
with the same compatibility row and no capacitor unknown. The $37$-artery
benchmark uses this outlet with $R_1=0$ and $R=R_2$.
In all cases, the local cell equation on the face still uses the HLL
flux. At a boundary, the trace value is prescribed by
$\widehat{\mathbf B}_h$; at an interior interface, conservation
determines the trace.

\subsection{Junctions: \texorpdfstring{$\widehat{\mathbf J}_h$}{J-hat}}\label{sec:junctions}
Consider a junction $J$ connecting $K\ge 2$ vessels. A node shared by two
elements of the same vessel, with continuous tube-law coefficients, is treated
as an interior face. A node shared by elements of two different vessels is a
two-way junction ($K=2$) and is closed by the junction conditions below.

Let
\[
\widehat{\mathbf{w}}_{h,i}
=
\begin{pmatrix}
\widehat{A}_i \\ \widehat{U}_i
\end{pmatrix},
\qquad i = 1,\ldots,K,
\]
denote the hybrid trace state of vessel $i$ at the junction, with
orientation sign
\[
s_i =
\begin{cases}
+1, & \text{if vessel } i \text{ ends at the junction},\\
-1, & \text{if vessel } i \text{ starts at the junction},
\end{cases}
\]
so that $s_i$ equals the outward orientation $\sigma$ of the incident element
at $J$. Following the standard $1\mathrm D$ junction closure
\cite{formaggia2003one,muller2016consistent}, mass conservation and
continuity of the total energy head $\mathcal H(\widehat{\mathbf{w}}_{h,i})=\tfrac12\widehat U_i^2+\Psi(\widehat A_i)$,
$\Psi=P/\rho$, give the $K$ constraints
\begin{equation}
\mathcal{F}_J
 (\widehat{\mathbf{w}}_{h,1},\ldots,\widehat{\mathbf{w}}_{h,K})
=
\begin{pmatrix}
\displaystyle\sum_{i=1}^{K}
s_i\,\widehat{A}_i\,\widehat{U}_i
\\[10pt]
\mathcal{H}(\widehat{\mathbf{w}}_{h,1})
- \mathcal{H}(\widehat{\mathbf{w}}_{h,2})
\\[4pt]
\vdots
\\[4pt]
\mathcal{H}(\widehat{\mathbf{w}}_{h,1})
- \mathcal{H}(\widehat{\mathbf{w}}_{h,K})
\end{pmatrix}
= \mathbf{0},
\qquad
\mathcal{H}(\widehat{\mathbf{w}}_{h,i})
=
\tfrac{1}{2}\widehat{U}_i^2 + \Psi(\widehat{A}_i),
\label{eq:junction_transmission}
\end{equation}
where $\Psi(A) = P(A)/\rho$. Row $1$ enforces global mass
conservation; rows $2$ to $K$ enforce continuity of total
head between all vessels and the reference vessel $i=1$.  The remaining $K$ equations impose outgoing characteristic compatibility on each branch.

\begin{remark}[Single-valued vs.\ multi-valued traces]
On interior faces and at physical boundaries the trace is single-valued, as in
standard HDG. At vessel junctions, each incident face carries its own trace
because the cross-sectional area may be discontinuous between branches. The
branch traces are coupled through conservation of mass, continuity of total
pressure, and outgoing characteristic compatibility.
\end{remark}
\begin{remark}[Volume balance]\label{rem:volume}
The first row of the junction conditions~\eqref{eq:junction_transmission}
conserves the trace flow at every junction,
$\sum_{i} s_i\,\widehat A_i\widehat U_i=0$. The network volume
\[
  V(t)=\int_{\Th}A_h\,\mathrm ds
\]
is therefore expected to satisfy
\begin{equation}
  \frac{\mathrm dV}{\mathrm dt}
  = Q_{\mathrm{in}}(t)-\sum_{k}\widehat Q_{\mathrm{out},k}(t)
  \label{eq:volume_balance}
\end{equation}
up to discretization error, where
$\widehat Q_{\mathrm{out},k}=\widehat A\widehat U$ is the trace flow at
terminal $k$. This balance is verified numerically in
Section~\ref{sec:liver}.
\end{remark}

\section{Solvability and energy balance}\label{sec:analysis}
\textit{Notation.}
Throughout this section $\mathbf w=(A,U)^\top$ with $A>0$, $Q=AU$
for the flow rate and $\widehat Q:=\widehat A\,\widehat U$ for its trace
counterpart. We set $\Psi(A):=P(A)/\rho$, $\sigma_{K,F}$ for the outward orientation of the face $F$ of $K$, and abbreviate
$\int_{\Th}(\cdot):=\sum_{K\in\Th}\int_K(\cdot)$. The admissible set is
\begin{equation}\label{eq:admissible}
  \mathcal U=\bigl\{(\wh,\who):\ A_h>\tfrac12A_{\min},\ \widehat A_h>\tfrac12A_{\min},\
  |U_h|<c_h,\ |\widehat U_h|<\widehat c_h\ \bigr\},
\end{equation}
the discrete states with strictly positive area and strictly
subcritical flow. Positive area keeps the tube law and wave speed differentiable, while subcriticality makes the entropy of Section~\ref{sec:entropy} convex.

\subsection{Conditional solvability}\label{sec:solvability}
For any $(\wh,\who)\in\mathbf W_h^k\times\mathbf M_h^k$ the semi-discrete HDG
scheme reads
\begin{subequations}
\begin{align}
\left(\partial_t\wh,\mathbf v\right)_K
- \left(\mathbf F(\wh),\bm b\cdot\nabla\mathbf v\right)_K
+ \left\langle\widehat{\mathbf F}_h\cdot\bm n,\mathbf v\right\rangle_{\partial K}
&= \left(\mathbf S(\wh),\mathbf v\right)_K,
\quad \forall\,\mathbf v\in(\mathcal P^k(K))^2,
\label{eq:semidiscrete1}\\
\sum_{K\in\Th}
\left\langle\widehat{\mathbf F}_h\cdot\bm n,\bm\mu\right\rangle
  _{\partial K\setminus(\Ehb\cup\Ehj)}
&+\sum_{F\in\Ehb}\left\langle\widehat{\mathbf B}_h,\bm\mu\right\rangle_F
\nonumber\\
&+\sum_{J\in\Ehj}\left\langle\widehat{\mathbf J}_h,\bm\mu\right\rangle_J=0,
\qquad \forall\,\bm\mu\in\mathbf M_h^k .
\label{eq:semidiscrete2}
\end{align}
\end{subequations}
Collecting the terminal
pressures into the differential block through $z=(W,P_c)$, the pair
\eqref{eq:semidiscrete1}--\eqref{eq:semidiscrete2} becomes a
differential--algebraic system,
\begin{equation}\label{dae_form}
\begin{aligned}
    M_z\,\dot z &= R(t,z,\widehat W),\qquad z=(W,P_c),\\
    0 &= G(t,z,\widehat W),
\end{aligned}
\end{equation}
in which $M_z=\mathrm{blockdiag}\bigl(M,\mathrm{diag}(C_i)\bigr)$ is positive
definite for $C_i>0$. The cell rows of $R$
contain the volume, flux and source terms, its terminal rows are
$\widehat Q_i-(P_{c,i}-P_{\mathrm{out},i})/R_{2,i}$, and the map $G$ collects the
transmission, boundary and junction constraints. Purely resistive terminals enter through algebraic boundary conditions.

\begin{lemma}(Residual regularity)\label{lem:reg}
The tube law~\eqref{eq:tube_law} with fixed positive wall and reference-area
coefficients is of class $C^2$ in area and satisfies $P'(A)>0$. For $C^1$
prescribed data, the maps $R$ and $G$ are of class
$C^1$ on the admissible set~\eqref{eq:admissible}, extended by arbitrary
capacitor pressures. In particular they are locally Lipschitz in the state.
\end{lemma}

\begin{proof}
The set $\mathcal U$ is open, since the inequalities are strict and
the discrete fields depend continuously on the coefficient vectors
$(W,\widehat W)$. On $\mathcal{U}$ the area is bounded below by $\frac{1}{2} A_{\mathrm{min}}>0$. Since the pressure law is $C^2$ in area with $P'(A)>0$, the wave speed $c= \sqrt{A\,P'/\rho}$ is $C^1$ in area. The extra derivative of the pressure law is needed because the wave speed depends on its first derivative. Subcriticality provides $s_L<0<s_R$, so only the central branch of \eqref{eq:hll_flux} is active and the denominator $s_R-s_L = 2\bar{c}$ is positive. Hence
$\widehat{\mathbf F}_h$ is a $C^1$ function of $(\wh,\who)$, and the same
reasoning applies to the boundary operator $\widehat{\mathbf B}_h$ and
the junction operator $\widehat{\mathbf J}_h$.

The maps $R$ and $G$ are finite sums of integrals of these functions against fixed basis functions, and are therefore $C^1$. A $C^1$ map is Lipschitz on every compact subset of $\mathcal{U}$.
\end{proof}

The local solvability result assumes an invertible trace Jacobian. Remark~\ref{rem:scope} describes its boundary, junction and interior blocks.

\begin{theorem}(Conditional local existence and uniqueness).
\label{thm:existence_uniqueness}
Assume that
\begin{itemize}
  \item[(i)] the pressure law $P(\cdot, s):(0,\infty)\to\mathbb R$ is of class $C^2$ in area with
    $P'(A)>0$, the prescribed data are $C^1$, and $C_i,R_{2,i}>0$ at the capacitive terminals;
  \item[(ii)] the initial approximation is admissible, $A_h(\cdot,0)\ge A_{\min}>0$
    and $|U_h(\cdot,0)|<c_h(\cdot,0)$, and admits a consistent trace
    $\widehat W_0$ with $G(0,z_0,\widehat W_0)=0$;
  \item[(iii)] the trace Jacobian
    $D_0:=\partial_{\widehat W}G(0,z_0,\widehat W_0)$ is nonsingular. 
\end{itemize}
Then there exists $T^*>0$ such that the DAE~\eqref{dae_form} possesses a unique
solution $(z,\widehat W)\in C^1([0,T^*])$ that remains admissible. Equivalently,
\eqref{eq:semidiscrete1}--\eqref{eq:semidiscrete2} admit a unique local solution
$(\wh,\who)$ in $\mathcal U$.
\end{theorem}

\begin{proof}
\emph{Step 1 (finite-dimensional DAE).}
The cell mass matrix $M=(\phi_i,\phi_j)_{\Th}$ is symmetric positive definite because the local
basis functions are linearly independent, and $\mathrm{diag}(C_i)$ is positive
definite for $C_i>0$. Hence $M_z$ is invertible.

\emph{Step 2 (elimination of the trace).}
By Lemma~\ref{lem:reg} the constraint map $G$ is $C^1$ in a neighbourhood of
$(0,z_0,\widehat W_0)$, and by hypothesis~(iii) its partial Jacobian $D_0$ is
invertible. The implicit function theorem gives a
neighbourhood of $(0,z_0)$ and a unique $\Phi\in C^1$ with
$\Phi(0,z_0)=\widehat W_0$ and  $G\bigl(t,z,\Phi(t,z)\bigr)=0$ .

\emph{Step 3 (reduced ODE).}
On that neighbourhood, 
\eqref{dae_form} reduces to
\[
  \dot z=M_z^{-1}R\bigl(t,z,\Phi(t,z)\bigr)=:\mathcal G(t,z),
  \qquad z(0)=z_0 .
\]
As a composition of $C^1$ maps with the constant matrix $M_z^{-1}$,
$\mathcal G$ is $C^1$ and hence locally Lipschitz in $z$. The Picard--Lindel\"of theorem
gives $T^*>0$ and a unique $z\in C^1([0,T^*])$. Setting
$\widehat W(t):=\Phi(t,z(t))$ yields a solution of \eqref{dae_form} \cite[Chapter~VI.1]{MR2657217}.

\emph{Step 4 (admissibility).}
Since $\mathcal U$ is open and the solution is continuous in time, it remains
admissible after shortening the interval if necessary.
\end{proof}

\begin{remark}[Trace Jacobian]\label{rem:scope}
The equations at each face involve only that face's traces, so the trace
Jacobian $D=\partial_{\widehat W}G$ is block diagonal, with one block per
interior face, boundary face, and junction. Hypothesis~(iii) therefore holds if
every block is nonsingular. The inflow and outlet blocks,
from~\eqref{eq:bc_inflow} and~\eqref{eq:bc_rcr}, have determinants
$\widehat U_h+\widehat c_h$ and
$P'(\widehat A_h)-R_1(\widehat U_h-\widehat c_h)$, both positive by
subcriticality. A linearisation of the junction conditions shows that the
junction blocks are nonsingular under the same assumption. The interior blocks
have eigenvalues $-(c\pm U)\neq0$ whenever the neighbouring cell states coincide
with the trace (Section~\ref{sec:interior}). In particular, (iii) holds at the
rest state used as initial condition in Section~\ref{sec:numerical_results}.
\end{remark}
\subsection{Entropy structure}\label{sec:entropy}

The system \eqref{equation_system} has a natural entropy,
which supplies both an admissibility criterion for weak solutions and the
analytical framework used below
\cite{tadmor1987numerical,harten1998convex}. We introduce a convex entropy
$\eta$, its entropy flux $q$ and the entropy variables $\vartheta=\eta'$ through
the usual compatibility relation with the physical flux. For the blood-flow model
the entropy has a transparent mechanical meaning: it is the sum of the kinetic
energy of the fluid column and the elastic energy stored in the vessel wall
\cite{formaggia2003one,MR4570554}. This structure yields positivity of the
discrete energy and, in Section~\ref{sec:energy_balance}, an exact balance for
the semi-discrete scheme.

\begin{definition}[Entropy pair]\label{def:entropy_defs}
A convex function $\eta:\mathcal U\to\mathbb R$ and a flux
$q:\mathcal U\to\mathbb R$ form an entropy pair for
$\partial_t\mathbf w+\partial_s\mathbf F(\mathbf w)=\mathbf S(\mathbf w)$ if
\[
  \eta'(\mathbf w)^\top\mathbf F'(\mathbf w)=q'(\mathbf w)^\top .
\]
The entropy variables are $\bm\vartheta(\mathbf w):=\eta'(\mathbf w)$. Smooth
solutions then satisfy the entropy balance
$\partial_t\eta+\partial_sq=\bm\vartheta^\top\mathbf S$ whenever the tube-law
coefficients are spatially constant. When the coefficients vary along the vessel, $\eta$ and
$q$ depend explicitly on $s$ and the identity acquires two extra terms,
\[
  \partial_t\eta+\partial_sq
  =\bm\vartheta^\top\mathbf S
   +(\partial_sq)_{\mathbf w}-\bm\vartheta^\top(\partial_s\mathbf F)_{\mathbf w};
\]
for the tube law \eqref{eq:tube_law} with constant $P^\circ$ these two terms
cancel, both being equal to $Q\,(\partial_sP)_A/\rho$.
\end{definition}

Let $P^\circ:=P_{0}+P_d$, so that $P(A_d,s)=P^\circ$ in every vessel
by \eqref{eq:tube_law}. Define the total head, the pressure potential and the
entropy by
\begin{equation}\label{eq:entropy}
  \mathcal H(\mathbf w):=\tfrac12U^2+\Psi(A),\qquad
  \Pi(A):=\frac1\rho\int_{A_d}^{A}\bigl(P(a)-P^\circ\bigr)\,da,\qquad
  \eta(\mathbf w):=\tfrac12AU^2+\Pi(A),
\end{equation}
with entropy flux $q(\mathbf w):=Q\bigl(\mathcal H-P^\circ/\rho\bigr)$.

\begin{lemma}(Entropy structure of the blood-flow system).\label{lem:entropy_pair}
The pair $(\eta,q)$ of \eqref{eq:entropy} is an entropy pair in the sense of
Definition~\ref{def:entropy_defs}, with entropy variables
$\bm\vartheta=\bigl(\mathcal H-P^\circ/\rho,\ Q\bigr)^\top$.
The Hessian of $\eta$ on $\mathcal U$ is
\begin{equation}\label{eq:hessian}
  \nabla^2\eta(\mathbf w)=\begin{pmatrix}c^2/A&U\\ U&A\end{pmatrix},
  \qquad \det\nabla^2\eta=c^2-U^2>0,
\end{equation}
so $\eta$ is strictly convex on $\mathcal{U}$, and $\mathbf S(\mathbf w)^\top\bm\vartheta(\mathbf w)=-fU^2\le 0.$
\end{lemma}

\begin{proof}
Since $\Pi'(A)=(P(A)-P^\circ)/\rho$, differentiation gives
$\eta_A=\tfrac12U^2+\Pi'(A)=\mathcal H-P^\circ/\rho$ and $\eta_U=Q$. With
$\mathbf F'=\bigl(\begin{smallmatrix}U&A\\ P'/\rho&U\end{smallmatrix}\bigr)$, a
direct computation gives $\partial_Aq=\vartheta_1U+QP'/\rho$ and
$\partial_Uq=\vartheta_1A+QU$, that is $q'=\bm\vartheta^\top\mathbf F'$, so
$(\eta,q)$ is an entropy pair.
The second derivatives are $\eta_{AA}=\Pi''(A)=P'(A)/\rho=c^2/A$, $\eta_{AU}=U$ and
$\eta_{UU}=A$. The Hessian has positive trace, and its determinant $c^2 -U^2$ is positive on $\mathcal{U}$, so it is positive definite. Finally
$\mathbf S^\top\bm\vartheta=(0,-fU/A)\cdot(\mathcal H-P^\circ/\rho,\,AU)^\top = -(fU/A)(AU)=-fU^2$.
\end{proof}
The two components of
$\bm\vartheta$ are the total head shifted by the common reference-pressure term
and the flow rate, respectively. The junction conditions
\eqref{eq:junction_transmission} conserve flow and impose continuity of total
head, hence also of its shifted value. Lemma~\ref{lemma:junction_energy} below
relies on this structure.

\begin{lemma}(Positivity of the energy).\label{energy_positivity}
Let
\begin{equation}\label{eq:energy}
  \mathscr E_h(t)=\int_{\Th}\eta(\wh)\,dx
  +\sum_{i\in\mathrm{RCR}}\frac{C_i}{2\rho}\bigl(P_{c,i}-P_{\mathrm{out},i}\bigr)^2 .
\end{equation}
Then $\mathscr E_h(t)\ge0$ at every admissible state, with equality only at the rest state $A=A_d$, $U=0$,
$P_{c,i}=P_{\mathrm{out},i}$.
\end{lemma}

\begin{proof}
By construction $\Pi(A_d)=0$, and $\Pi'(A_d)=0$ because $P(A_d)=P^\circ$. Since $\Pi''(A)=P'(A)/\rho>0,\qquad A>0$, $\Pi$ is strictly convex with its unique minimum at $A_d$, so $\Pi \ge 0$ with equality only
at $A=A_d$. Since $\tfrac12AU^2\ge0$ for $A>0$, vanishing only when $U=0$, we
obtain $\eta\ge0$ with the stated equality case. The Windkessel contributions in
\eqref{eq:energy} are squares with $C_i\ge0$, and $\mathscr E_h\ge0$ follows by
integration and summation.
\end{proof}

\subsection{Discrete energy balance}\label{sec:energy_balance}
We assume exact spatial integration and a constant reference pressure $P^\circ$. Write $f_{K,F}:=\widehat{\mathbf F}_h\cdot\bm n$ for the numerical normal flux on the face $F$ of the element $K$.
\begin{lemma}(Discrete energy identity).\label{lemma:eta_inequality}
Let $(\wh,\who)\in\mathcal U$ solve \eqref{eq:semidiscrete1}--\eqref{eq:semidiscrete2}
with polynomial degree $p=k$, and let $v_h$ denote the elementwise $L^2$
projection of $\bm\vartheta(\wh,s)$ onto $\mathbf W_h^k$. Define 
\begin{align}
  R_h&:=\sum_{K}\sum_{F\subset\partial K}
    \Bigl[\,v_h(F)^\top f_{K,F}
      -\sigma_{K,F}\,q(\widehat{\mathbf w}_{K,F},F)\Bigr],
  \label{eq:Rh}\\[2pt]
  V_h&:=\sum_{K}\Bigl[\bigl(\mathbf F(\wh,s),\partial_sv_h\bigr)_K
      +\bigl(\mathbf S,\,v_h-\bm\vartheta(\wh,s)\bigr)_K\Bigr].
  \label{eq:Vh}
\end{align}
Then 
\begin{equation}\label{eq:semi_discrete_energy}
  \frac{d}{dt}\int_{\Th}\eta(\wh,s)
  +\int_{\Th}fU_h^2
  +\sum_{K}\sum_{F\subset\partial K}\sigma_{K,F}\,q(\widehat{\mathbf w}_{K,F},F)
  +R_h=V_h .
\end{equation}
For $p=0$ with coefficients constant on each element one has $v_h=\bm\vartheta(\wh)$
and $V_h=0$.
\end{lemma}

\begin{proof}
Test \eqref{eq:semidiscrete1} with $\mathbf v=v_h$. Since
$\partial_t\wh\in\mathbf W_h^k$, projection orthogonality gives
\[
  \bigl(\partial_t\wh,v_h\bigr)_K=\bigl(\partial_t\wh,\bm\vartheta\bigr)_K
  =\int_K\partial_t\eta=\frac{d}{dt}\int_K\eta ,
\]
while $\bm\vartheta^\top\mathbf S=-fU_h^2$ from Lemma~\ref{lem:entropy_pair} gives
$(\mathbf S,v_h)_K=-\int_KfU_h^2+(\mathbf S,v_h-\bm\vartheta)_K$. Summing over $K$ and collecting the volume and projection terms in
$V_h$ gives
\[
  \frac{d}{dt}\int_{\Th}\eta+\int_{\Th}fU_h^2
  +\sum_{K,F}v_h(F)^\top f_{K,F}=V_h .
\]
Adding and subtracting the physical trace entropy fluxes
$\sigma_{K,F}\,q(\widehat{\mathbf w}_{K,F})$ in the face sum yields \eqref{eq:semi_discrete_energy}.

For $p=0$ with elementwise constant coefficients the function
$\bm\vartheta(\wh,s)$ already lies in $\mathbf W_h^0$, so $v_h=\bm\vartheta(\wh)$,
$\partial_sv_h=0$ and $v_h=\bm\vartheta$, and $V_h=0$.
\end{proof}

The term $R_h$ contains the complete HLL flux, including its physical-flux part, and has no
sign in general.

\begin{lemma}(Junction energy flux).\label{lemma:junction_energy}
Under the junction conditions \eqref{eq:junction_transmission},
\[
  \Sigma_J:=\sum_{i\in J}s_i\,q(\widehat{\mathbf w}_i)=0 .
\]
\end{lemma}

\begin{proof}
All branches share one total head value
$\mathcal H(\widehat{\mathbf w}_i)\equiv\mathcal H^*$, so
\[
  \Sigma_J=\bigl(\mathcal H^*-P^\circ/\rho\bigr)\sum_{i\in J}s_i\widehat Q_i
  =\bigl(\mathcal H^*-P^\circ/\rho\bigr)\sum_{i\in J}s_i\widehat A_i\widehat U_i=0 .
  \qedhere
\]
by the conservation of trace flow.
\end{proof}

\begin{theorem}(Energy balance).\label{thm:energy_balance}
On the interval $[0, T^*]$ of Theorem~\ref{thm:existence_uniqueness},
assume prescribed-flow inlets
and capacitive $\mathrm{RCR}$ outlets $P_{\mathrm{out},i}$, and the junction conditions  \eqref{eq:junction_transmission}. Then
\begin{equation}\label{eq:energy_balance}
  \frac{d}{dt}\mathscr E_h
  +\int_{\Th}fU_h^2
  +\sum_{i\in\mathrm{RCR}}
    \left[\frac{R_{1,i}\widehat Q_i^{\,2}}{\rho}
      +\frac{(P_{c,i}-P_{\mathrm{out},i})^2}{\rho R_{2,i}}\right]
  =\mathcal B_\partial+V_h-R_h ,
\end{equation}
where the friction and resistive terms on the left are nonnegative and
\begin{equation}\label{eq:boundary_term}
  \mathcal B_\partial
  =\sum_{\mathrm{inlets}}Q_{\mathrm{in}}
     \bigl(\widehat{\mathcal H}-P^\circ/\rho\bigr)
   -\sum_{i\in\mathrm{RCR}}\widehat Q_i
     \left[\tfrac12\widehat U_i^2
       +\frac{P_{\mathrm{out},i}-P^\circ}{\rho}\right] .
\end{equation}
$V_h-R_h$ collect the boundary power and numerical-flux and volume contributions. Resistive outlets are included by setting $C_i = 0$, replacing $R_{1,i}$ by $R_i$, and dropping the
capacitor terms.
\end{theorem}
\begin{proof}
We evaluate the face sum in \eqref{eq:semi_discrete_energy}. On an interior face the trace is single-valued and
$\sigma_L=-\sigma_R$, so the two contributions cancel. At a junction
they cancel by Lemma~\ref{lemma:junction_energy}. At an inlet, where
$\sigma=-1$, the contribution is
$-Q_{\mathrm{in}}\bigl(\widehat{\mathcal H}-P^\circ/\rho\bigr)$.
At an $\mathrm{RCR}$ outlet, we expand the flux at the $\mathrm{RCR}$ face,
where $\sigma=+1$, and insert the pressure-continuity row
$P(\widehat A_i)=R_{1,i}\widehat Q_i+P_{c,i}$ of \eqref{eq:bc_rcr}. This gives
\begin{equation}\label{eq:out1}
  \widehat Q_i\bigl(\widehat{\mathcal H}_i-P^\circ/\rho\bigr)
  =\widehat Q_i\left(\tfrac12\widehat U_i^2
     +\frac{P(\widehat A_i)-P^\circ}{\rho}\right)
  =\tfrac12\widehat Q_i\widehat U_i^2
   +\frac{R_{1,i}\widehat Q_i^{\,2}}{\rho}
   +\frac{\widehat Q_i\bigl(P_{c,i}-P^\circ\bigr)}{\rho}.
\end{equation}
Multiplying the capacitor equation
$C_i\dot P_{c,i}=\widehat Q_i-(P_{c,i}-P_{\mathrm{out},i})/R_{2,i}$ by
$(P_{c,i}-P_{\mathrm{out},i})/\rho$, and using
$(P_{c,i}-P_{\mathrm{out},i})\dot P_{c,i}
=\tfrac12\tfrac{d}{dt}(P_{c,i}-P_{\mathrm{out},i})^2$ for constant
$P_{\mathrm{out},i}$, gives the energy law of the terminal,
\begin{equation}\label{eq:out2}
  \frac{d}{dt}\left[\frac{C_i}{2\rho}(P_{c,i}-P_{\mathrm{out},i})^2\right]
  +\frac{(P_{c,i}-P_{\mathrm{out},i})^2}{\rho R_{2,i}}
  =\frac{\widehat Q_i}{\rho}\bigl(P_{c,i}-P_{\mathrm{out},i}\bigr).
\end{equation}
That is, storage rate plus resistive loss equals the power received from the
vessel. Solving \eqref{eq:out2} for $\widehat Q_iP_{c,i}/\rho$ and substituting
into \eqref{eq:out1},
\begin{equation}\label{eq:out3}
  \widehat Q_i\bigl(\widehat{\mathcal H}_i-P^\circ/\rho\bigr)
  =\frac{d}{dt}\!\left[\frac{C_i(P_{c,i}-P_{\mathrm{out},i})^2}{2\rho}\right]
   +\underbrace{\frac{R_{1,i}\widehat Q_i^{\,2}}{\rho}
     +\frac{(P_{c,i}-P_{\mathrm{out},i})^2}{\rho R_{2,i}}}_{\ \ge0}
   +\widehat Q_i\left[\tfrac12\widehat U_i^2
     +\frac{P_{\mathrm{out},i}-P^\circ}{\rho}\right].
\end{equation}
The first term is the time derivative of the terminal part of $\mathscr E_h$ in
\eqref{eq:energy} and is absorbed into $\tfrac{d}{dt}\mathscr E_h$. The second is
nonnegative and stays on the left of \eqref{eq:energy_balance}. The third,
together with the inlet contribution, forms
$\mathcal B_\partial$ as given in \eqref{eq:boundary_term}. Substituting into
\eqref{eq:semi_discrete_energy} proves the balance.
\end{proof}

\begin{corollary}(Closed network)\label{closed_network}
If $\Ehb=\emptyset$, then $\mathcal B_\partial\equiv0$ and
\[
  \frac{d}{dt}\mathscr E_h+\int_{\Th}fU_h^2=V_h-R_h,
\]
so the discrete energy is non-increasing provided $R_h-V_h\ge0$.
\end{corollary}

\begin{corollary}(Energy bound)\label{given_inflow}
Under the hypotheses of Theorem~\ref{thm:energy_balance},
\begin{equation}\label{standard_energy}
  \mathscr E_h(t)\le\mathscr E_h(0)
  +\int_0^t\Bigl(|\mathcal B_\partial|+|V_h|+|R_h|\Bigr)\,d\tau,
  \qquad 0\le t\le T^* .
\end{equation}
\end{corollary}
\begin{proof}
Discard the nonnegative friction and resistive terms in
\eqref{eq:energy_balance} and integrate in time.
\end{proof}
The bound \eqref{standard_energy} is not an a priori estimate: $\mathcal{B}_{\partial}$, $V_h$ and $R_h$ depend on the computed boundary state and on the solution itself.

\begin{remark}(Gravitational source)\label{rem:gravity}
When a gravitational body force is included, the momentum source carries the
extra term $\bm g\cdot\bm b(\bm x)$, contributing the power
$\int_{\Th}(\bm g\cdot\bm b)\,Q_h$ to the right-hand side of
\eqref{eq:semi_discrete_energy} and \eqref{eq:energy_balance}, where it is
collected in $V_h$. Its absolute value contributes to the right-hand side of \eqref{standard_energy}.
\end{remark}
\section{Time Integration}\label{sec:time_integration}

The semi-discrete problem has the residual DAE form
\begin{equation}
  \mathbf F(t,\mathbf y,\dot{\mathbf y})=\mathbf 0
  \label{eq:time_dae}
\end{equation}
with the differential and algebraic components identified below. Let $t_n$ be
the current time, $h_n=t_n-t_{n-1}$ the current step size, and $q_n$ the
order of a variable-step BDF formula. The multistep approximation satisfies
\begin{equation}
  \sum_{i=0}^{q_n}\alpha_{n,i}\mathbf y_{n-i}
  = h_n\dot{\mathbf y}_n,
  \label{eq:bdf_formula}
\end{equation}
where the coefficients $\alpha_{n,i}$ are determined by the order and the
history of the step sizes. Substitution into the residual DAE produces the
nonlinear algebraic problem
\begin{equation}
  \mathbf G_n(\mathbf y_n)
  \equiv
  \mathbf F\!\left(t_n,\mathbf y_n,
  \frac{1}{h_n}\sum_{i=0}^{q_n}\alpha_{n,i}\mathbf y_{n-i}\right)
  =\mathbf 0 .
  \label{eq:bdf_residual}
\end{equation}
Newton linearization of this problem gives the following result.
\begin{equation}
  \mathcal J_n
  =\frac{\partial\mathbf G_n}{\partial\mathbf y_n}
  =\frac{\partial\mathbf F}{\partial\mathbf y}
   +\gamma_n\frac{\partial\mathbf F}{\partial\dot{\mathbf y}},
  \qquad
  \gamma_n=\frac{\alpha_{n,0}}{h_n}.
  \label{eq:newton_linearization}
\end{equation}
This relation is the mathematical link between the BDF time discretization and
the global Newton system. The numerical computations use the SUNDIALS IDA
solver as the backend for this residual DAE procedure \cite{hindmarsh2005sundials}.

\paragraph{Global unknown vector.}
The unknowns are grouped into a cell block, a trace block, and a capacitor
block,
\begin{equation}
  \mathbf y
  =
  \bigl(\mathbf y_{\mathrm{cell}},\;
        \mathbf y_{\mathrm{trace}},\;
        \mathbf y_{\mathrm{pc}}\bigr),
  \label{eq:global_state}
\end{equation}
\begin{equation}
  \mathbf y_{\mathrm{pc}}
  =\bigl(P_{c,1},\dots,P_{c,N_{\mathrm{rcr}}}\bigr),
  \label{eq:capacitor_state}
\end{equation}
where $\mathbf y_{\mathrm{pc}}$ carries one capacitor pressure $P_{c,i}$ for
each \emph{capacitive} ($C_i>0$) $\mathrm{RCR}$ terminal. Purely resistive ($C_i=0$)
outlets impose pressure through algebraic conditions in $\mathbf R_{\mathrm{trace}}$.

\paragraph{Three-block residual.}
Stacking the cell weak form~\eqref{eq:hdg_local}, the transmission conditions,
and the Windkessel residual~\eqref{eq:rcr_capacitor} gives
\begin{equation}
  \mathbf F(t,\mathbf y,\dot{\mathbf y})
  =
  \begin{pmatrix}
  M_K\,\dot{\mathbf y}_{\mathrm{cell}}-\mathbf R_{\mathrm{cell}}(t,\mathbf y)\\[6pt]
  \mathbf R_{\mathrm{trace}}(t,\mathbf y)\\[6pt]
  \mathbf R_{\mathrm{pc}}(t,\mathbf y,\dot{\mathbf y}_{\mathrm{pc}})
  \end{pmatrix}
  =\mathbf 0 .
  \label{eq:global_dae_residual}
\end{equation}
Here $\mathbf R_{\mathrm{cell}}$ is the cell residual associated with the local weak form~\eqref{eq:hdg_local} (volume
flux, face flux via the trace, and source), $\mathbf R_{\mathrm{trace}}$
collects the interior, inflow, single-resistance, RCR, and junction transmission
residuals, and $\mathbf R_{\mathrm{pc}}$ is the capacitor
residual~\eqref{eq:rcr_capacitor}, which contains the time derivative
$C_i\,\dot P_{c,i}$. Thus, the time derivative enters only through the cell mass
matrix $M_K$ and the capacitor block; the trace rows are purely algebraic.

\textit{Differential / algebraic split.}
The rows carrying a time derivative are differential; the rest are algebraic:
\begin{equation}
  \text{differential:} \;
  \{\mathbf y_{\mathrm{cell}}\}\cup\{\mathbf y_{\mathrm{pc}}\},
  \label{eq:differential_components}
\end{equation}
\begin{equation}
  \text{algebraic:}  \;
  \{\mathbf y_{\mathrm{trace}}\}.
  \label{eq:algebraic_components}
\end{equation}
Under the non-singularity condition established in Section~\ref{sec:analysis},
the trace variables can be locally eliminated, so the residual
system~\eqref{eq:global_dae_residual} is a semi-explicit index-one DAE.

\textit{Newton system.}
Using the notation in \eqref{eq:newton_linearization}, the derivative block is
\begin{equation}
  \frac{\partial\mathbf F}{\partial\dot{\mathbf y}}
  =\mathrm{blockdiag}\bigl(M_K,\,\mathbf 0,\,\mathbf C_{\mathrm{wk}}\bigr),
  \qquad
  \mathbf C_{\mathrm{wk}}=\mathrm{diag}(C_i).
  \label{eq:dae_derivative_block}
\end{equation}
The Newton matrix has the form 
\begin{equation}
 \mathcal J_n
  =
  \begin{pmatrix}
   \gamma_n M_K + A & B & 0\\[4pt]
  C & D & E\\[4pt]
  0 & G & \gamma_n \,\mathbf C_{\mathrm{wk}}+\mathbf R_2^{-1}
  \end{pmatrix},
  \label{eq:ida_Jblocks}
\end{equation}
in which $A,B,C,D$ are the cell--cell, cell--trace, trace--cell and
trace--trace blocks already assembled for the network. The capacitor coupling reuses the boundary-section derivatives directly:

\begin{equation}
  E \;\text{contains}\; E_{\mathrm{rcr}}
  =\frac{\partial\widehat{\mathbf B}_h^{\mathrm{rcr}}}{\partial P_c}
  =\begin{bmatrix}-1\\[2pt]0\end{bmatrix} ,
  \label{eq:rcr_trace_derivative}
\end{equation}
\begin{equation}
  G \;\text{contains}\;
  \frac{\partial R_{\mathrm{pc},i}}{\partial(\widehat A_i,\widehat U_i)}
  =\bigl(-\widehat U_i,\;-\widehat A_i\bigr),
  \label{eq:rcr_capacitor_derivative}
\end{equation}
and the capacitor diagonal entry per outlet is
\begin{equation}
  \frac{\partial R_{\mathrm{pc},i}}{\partial P_{c,i}}
  +\gamma_n\,\frac{\partial R_{\mathrm{pc},i}}{\partial\dot P_{c,i}}
  =\frac{1}{R_{2,i}}+\gamma_n\,C_i .
  \label{eq:capacitor_newton_diagonal}
\end{equation}
The two mass terms $\gamma_n M_K$ and $\gamma_n C_i$, together with the remaining entries, come from the residual Jacobian assembly. The capacitance $C_i$ is distinct from the trace--cell block $C$.

This block structure is also the coupling interface for an external dynamical
model. If an externally coupled discretization introduces unknowns $z$ and
residuals $R_{\mathrm{ext}}(t, z, \dot{z}; y)$, the coupled problem is obtained
by augmenting $ \bm y$  and $\bm F$ with the variables $z$ and the interface
equations. The resulting off-diagonal blocks are assembled in the same Newton matrix \eqref{eq:ida_Jblocks}, providing a full monolithic formulation for the coupled problem.

\textit{Solver.}
The complete solution procedure is summarized in
Figure~\ref{fig:solver_workflow}. Each Newton iteration solves $\mathcal{J}_n \,\delta\mathbf y=-\mathbf G_n$
with the sparse direct solver \textsc{mumps} \cite{amestoy2000mumps}. Relative to the explicit-update formulation only the system size and
sparsity change: a few extra rows/columns for $y_{\mathrm{pc}}$ and the off-diagonal $E$, $G$ couplings.

\noindent
\begin{figure}[t]
\centering
\resizebox{\textwidth}{!}{
\includegraphics{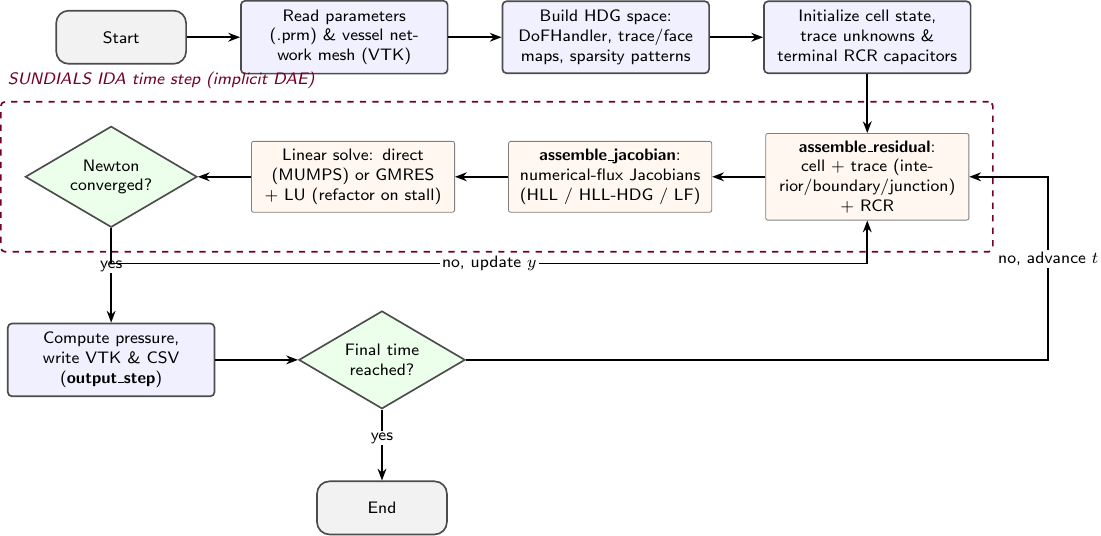}} \caption[Workflow of the implicit HDG
  solver]{Workflow of the implicit HDG solver. After the parameters and the
  vessel-network mesh are read, the HDG space is built and the cell, trace and
  terminal $\mathrm{RCR}$ capacitor unknowns are initialized. The nonlinear solve within each SUNDIALS IDA
  time step (dashed box) consists of residual assembly, Jacobian assembly,
  and a sparse linear solve, repeated until convergence. Jacobian assembly
  may be delayed, so the nonlinear solve does not necessarily use exact
  Newton iterations. Several heuristics are used to minimize computational
  cost, with new Jacobians assembled only when the available Jacobian does
  not produce a convergent sequence. At constant time intervals, the pressure
  is computed and the solution is written to file. The time loop continues
  until the final time is reached.}
\label{fig:solver_workflow}
\end{figure}

\section{Numerical Results}\label{sec:numerical_results} All computations and
scaling results reported in this section are obtained with the finite element
library \texttt{deal.II}~\cite{ArndtBangerthBlais-2026-a,ArndtBangerthDavydov-2021-a} (version 9.8.0), built with GCC
13.3.0, Open MPI 4.1.6 and SUNDIALS 6.4.1. The experiments are run on a
\ScalingMachineDescription.
\subsection{Convergence Studies}
As a first verification test, we consider a manufactured solution to assess
the spatial accuracy of the proposed HDG method. The manufactured test case
is inspired by Example~5.1 of \cite{lucca2023semi}. We consider the
one-dimensional blood-flow equations~\eqref{equation_system} on a single
vessel of length $L=1\,\mathrm{m}$ and final time $T=0.1\,\mathrm{s}$. The
exact manufactured solution is chosen as
\begin{align*}
    A(x,t) &= A_d + \alpha_A\,(1 - \cos(2 \pi t))\, \sin(2 \pi x/L), \\
    U(x,t) &= \alpha_U\, \sin(2 \pi t)\, \cos(2 \pi x/L),
\end{align*}
where $\alpha_A = 5\times 10^{-5}\,\mathrm{m}^2$ and
$\alpha_U = 0.5\,\mathrm{m/s}$. The corresponding source terms are obtained
by substituting these expressions into \eqref{equation_system}, so that the
manufactured solution satisfies the governing equations exactly. Unless
otherwise stated, all physical parameters are given in SI units. The
following parameters are used in the numerical experiments:

\emph{blood density} $\rho=1060$, \emph{viscosity} $\mu=0$,
\emph{wall thickness} $h_\mathrm{w}=1.2\times10^{-3}$,
\emph{Young's modulus} $E=4\times10^5$, \emph{reference pressure}
$P_0=0$, \emph{diastolic pressure} $P_d=9.46\times10^3$, and
\emph{diastolic cross-sectional area} $A_d=4.5239\times10^{-4}$. The errors
in the variables $A$ and $U$ are measured using the
$L^2$ norm and broken $H^1$ seminorm,
\begin{equation*}
\|e_w\|_{L^2}=\|w-w_h\|_{L^2},
\qquad
|e_w|_{H^1}=|w-w_h|_{H^1},
\qquad w\in\{A,U\}.
\end{equation*}
The corresponding errors and observed convergence rates for polynomial degrees $p=1$, $2$ and $p=3$ are reported in
Table~\ref{tab:convergence}. The expected order of convergence is achieved
for both quantities $A$ and $U$ in the $L^2$ norm and broken $H^1$ seminorm.
\begin{table}[ht]
\centering
\caption{Spatial convergence of the HDG scheme on a manufactured solution.
         Errors and convergence rates in the $L^2$ norm and broken $H^1$ seminorm for the cross-sectional
         area $A$ and velocity $U$, for polynomial degrees $p=1,2,3$.}
\label{tab:convergence}
\setlength{\tabcolsep}{5pt}
\renewcommand{\arraystretch}{1.05}
\begin{tabular}{c c rr rr rr rr}
\hline\hline
$p$ & $N_h$ &
$\|e_A\|_{L^2}$ & rate &
$|e_A|_{H^1}$   & rate &
$\|e_U\|_{L^2}$ & rate &
$|e_U|_{H^1}$   & rate \\
\hline
\multirow{5}{*}{1}
 & 4  & $2.197\times10^{-6}$ & --   & $3.316\times10^{-5}$ & --   & $1.926\times10^{-2}$ & --   & $5.766\times10^{-1}$ & --   \\
 & 8  & $5.714\times10^{-7}$ & 1.94 & $1.758\times10^{-5}$ & 0.92 & $5.021\times10^{-3}$ & 1.94 & $2.982\times10^{-1}$ & 0.95 \\
 & 16 & $1.441\times10^{-7}$ & 1.99 & $8.960\times10^{-6}$ & 0.97 & $1.286\times10^{-3}$ & 1.97 & $1.506\times10^{-1}$ & 0.99 \\
 & 32 & $3.609\times10^{-8}$ & 2.00 & $4.514\times10^{-6}$ & 0.99 & $3.256\times10^{-4}$ & 1.98 & $7.549\times10^{-2}$ & 1.00 \\
 & 64 & $9.030\times10^{-9}$ & 2.00 & $2.264\times10^{-6}$ & 1.00 & $8.193\times10^{-5}$ & 1.99 & $3.778\times10^{-2}$ & 1.00 \\[3pt]
\midrule
\multirow{5}{*}{2}
 & 4  & $2.729\times10^{-7}$ & --   & $8.754\times10^{-6}$ & --   & $2.554\times10^{-3}$ & --   & $1.295\times10^{-1}$ & --   \\
 & 8  & $3.502\times10^{-8}$ & 2.96 & $2.289\times10^{-6}$ & 1.94 & $3.328\times10^{-4}$ & 2.94 & $3.343\times10^{-2}$ & 1.95 \\
 & 16 & $4.409\times10^{-9}$ & 2.99 & $5.816\times10^{-7}$ & 1.98 & $4.238\times10^{-5}$ & 2.97 & $8.438\times10^{-3}$ & 1.99 \\
 & 32 & $5.523\times10^{-10}$ & 3.00 & $1.463\times10^{-7}$ & 1.99 & $5.343\times10^{-6}$ & 2.99 & $2.116\times10^{-3}$ & 2.00 \\
 & 64 & $6.908\times10^{-11}$ & 3.00 & $3.668\times10^{-8}$ & 2.00 & $6.707\times10^{-7}$ & 2.99 & $5.297\times10^{-4}$ & 2.00 \\[3pt]
 \midrule
\multirow{5}{*}{3}
 & 4  & $2.621\times10^{-8}$ & --   & $1.386\times10^{-6}$ & --   & $2.557\times10^{-4}$ & --   & $1.895\times10^{-2}$ & --   \\
 & 8  & $1.676\times10^{-9}$ & 3.97 & $1.796\times10^{-7}$ & 2.95 & $1.657\times10^{-5}$ & 3.95 & $2.438\times10^{-3}$ & 2.96 \\
 & 16 & $1.054\times10^{-10}$ & 3.99 & $2.274\times10^{-8}$ & 2.98 & $1.050\times10^{-6}$ & 3.98 & $3.075\times10^{-4}$ & 2.99 \\
 & 32 & $6.598\times10^{-12}$ & 4.00 & $2.855\times10^{-9}$ & 2.99 & $6.602\times10^{-8}$ & 3.99 & $3.855\times10^{-5}$ & 3.00 \\
 & 64 & $4.299\times10^{-13}$ & 3.94 & $3.574\times10^{-10}$ & 3.00 & $4.267\times10^{-9}$ & 3.95 & $4.825\times10^{-6}$ & 3.00 \\
\hline
\end{tabular}
\end{table}

\subsection{Verification with benchmark results}
\label{sec:verification}
In this section, we verify the proposed implicit HDG scheme against the
one-dimensional arterial blood-flow benchmark of Boileau et al.~\cite{boileau2015benchmark}. Of the test cases proposed there, we consider two: a single aortic bifurcation and a $37$-artery network for which in~vitro pressure and flow
measurements are available. These two cases are chosen because together they cover both types of boundary conditions used in this work: the aortic bifurcation is closed with full three-element ($\mathrm{RCR}$) Windkessel models, whereas the $37$-artery network is terminated by single resistances. We reproduce both cases with the HDG method, using the polynomial
degrees and mesh refinements specified in Table~1 of
\cite{boileau2015benchmark}. The initial area $A_0$ \cite{boileau2015benchmark} that depends on the diastolic area $A_d$ at $P= P_d$ is calculated using 
$$ A_0 = A_d \left( 1- \sqrt{A_d} \frac{P_d}{\beta}\right).$$

In both cases, the governing $1\mathrm{D}$
equations are closed with the tube law~\eqref{eq:tube_law}. All geometric, wall and terminal parameters, together with the prescribed inflow $Q_{\mathrm{in}}(t)$, are taken directly from the supplementary data of \cite{boileau2015benchmark}.

\subsubsection{Aortic bifurcation}
\label{sec:aortic_bifurcation}
The first test case is a symmetric, single-bifurcation model of the abdominal
aorta dividing into the two iliac arteries~\cite{boileau2015benchmark}. Each
iliac outlet is coupled to a three-element ($\mathrm{RCR}$) Windkessel model
representing the remainder of the systemic circulation. The outlet condition
is imposed through the terminal operator, while the aortic inlet is driven by
the \emph{in vivo} flow-rate signal $Q_{\mathrm{in}}(t)$ from the supplementary
data of the benchmark. In this case, we consider initial parameters as area $(A(x,0), U(x,0), P(x,0)) = (A_0, U_0, P_0)$ \cite[Table VII]{boileau2015benchmark}.

Figures~\ref{fig:aorta_results}--\ref{fig:iliac_results} compare the HDG waveforms of pressure, flow rate and luminal-radius change against the discontinuous Galerkin method (DCG) reference solutions at three locations: the midpoint of the aorta, the aortic bifurcation (junction point) and the midpoint of one iliac artery over the final periodic cycle, from $T = 28.6\,\mathrm{s}$ to $29.7\,\mathrm{s}$. At every location the HDG waveforms are visually indistinguishable from the
reference DCG results, reproducing both the systolic pressure peak and the reflected wave in diastole.
Following the benchmark configuration, we use polynomial degree $3$ with no
mesh refinement and employ $\mathrm{BDF}$ of maximum order $2$. Running to a
final time of $30\,\mathrm{s}$ our solver completes the
simulation in $8.741\,\mathrm{s}$ on a single processor, including the time spent in writing $6000$ output files.

\begin{figure*}[t]
    \centering

    \begin{subfigure}[b]{0.32\textwidth}
        \centering
        \includegraphics[width=\linewidth]{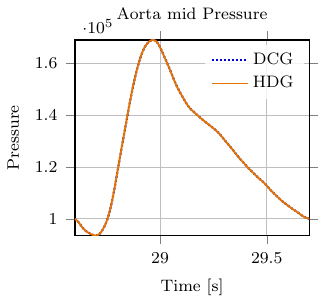}
        \caption{Pressure $(\mathrm{dyn/cm^2})$}
        \label{fig:aorta_pressure}
    \end{subfigure}
    \hfill
    \begin{subfigure}[b]{0.32\textwidth}
        \centering
        \includegraphics[width=\linewidth]{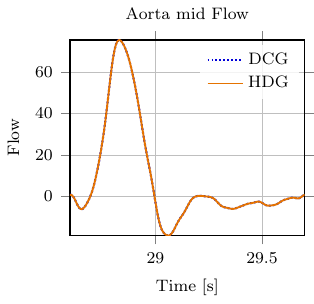}
        \caption{Flow rate $(\mathrm{ml/s})$}
        \label{fig:aorta_flow}
    \end{subfigure}
    \hfill
    \begin{subfigure}[b]{0.32\textwidth}
        \centering
        \includegraphics[width=\linewidth]{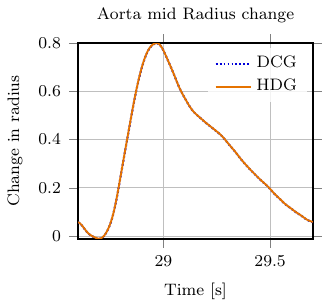}
        \caption{Radius change $(\mathrm{mm})$}
        \label{fig:aorta_radius}
    \end{subfigure}

    \caption{Comparison of the DCG and HDG solutions at the midpoint of the
    aorta. The pressure, flow rate, and radius change are shown as functions of
    time.}
    \label{fig:aorta_results}
\end{figure*}

\begin{figure*}[t]
    \centering

    \begin{subfigure}[b]{0.32\textwidth}
        \centering
        \includegraphics[width=\linewidth]{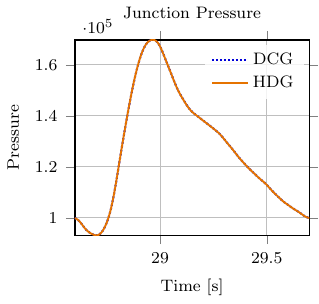}
        \caption{Pressure $(\mathrm{dyn/cm^2})$}
        \label{fig:junction_pressure}
    \end{subfigure}
    \hfill
    \begin{subfigure}[b]{0.32\textwidth}
        \centering
        \includegraphics[width=\linewidth]{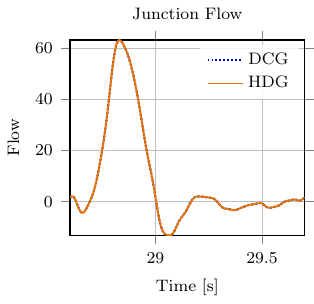}
        \caption{Flow rate $(\mathrm{ml/s})$}
        \label{fig:junction_flow}
    \end{subfigure}
    \hfill
    \begin{subfigure}[b]{0.32\textwidth}
        \centering
        \includegraphics[width=\linewidth]{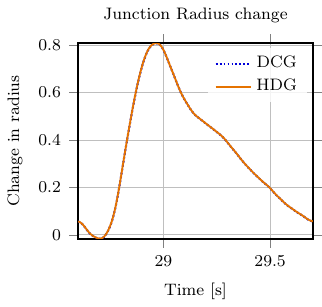}
        \caption{Radius change $(\mathrm{mm})$}
        \label{fig:junction_radius}
    \end{subfigure}

    \caption{Comparison of the DCG and HDG solutions at the junction. The
    pressure, flow rate, and radius change are shown as functions of time.}
    \label{fig:junction_results}
\end{figure*}

\begin{figure*}[t]
    \centering

    \begin{subfigure}[b]{0.32\textwidth}
        \centering
        \includegraphics[width=\linewidth]{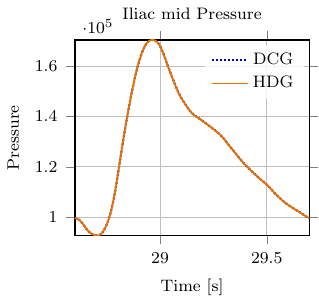}
        \caption{Pressure $(\mathrm{dyn/cm^2})$}
        \label{fig:iliac_pressure}
    \end{subfigure}
    \hfill
    \begin{subfigure}[b]{0.32\textwidth}
        \centering
        \includegraphics[width=\linewidth]{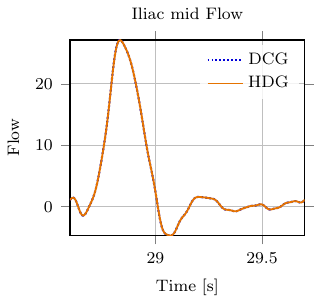}
        \caption{Flow rate $(\mathrm{ml/s})$}
        \label{fig:iliac_flow}
    \end{subfigure}
    \hfill
    \begin{subfigure}[b]{0.32\textwidth}
        \centering
        \includegraphics[width=\linewidth]{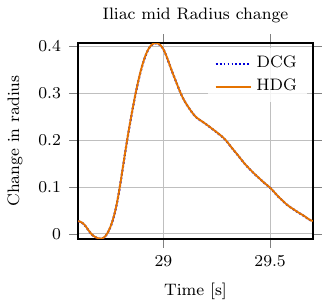}
        \caption{Radius change $(\mathrm{mm})$}
        \label{fig:iliac_radius}
    \end{subfigure}

    \caption{Comparison of the DCG and HDG solutions at the midpoint of the
    iliac artery. The pressure, flow rate, and radius change are shown as
    functions of time.}
    \label{fig:iliac_results}
\end{figure*}

\subsubsection{Thirty-seven-artery network}
\label{sec:37artery}

The second test is the $37$-artery network of the largest central systemic
arteries, realised in~vitro as a set of silicone tubes and used as a
$1\mathrm{D}$
benchmark in~\cite{boileau2015benchmark}. At the inlet of the ascending aorta
the flow rate measured in~vitro is prescribed as $Q_{\mathrm{in}}(t)$, and every
terminal vessel is closed with a single-resistance outlet, i.e. the $C\to0$,
$R_1\to0$ limit of the Windkessel operator of Section~\ref{sec:boundaries}, with
$R_2$ the terminal resistance of the benchmark. 

Following the benchmark configuration~\cite[Table~1]{boileau2015benchmark}, we
use polynomial degree $3$ and employ $\mathrm{BDF}$ of maximum order $3$.
Running to a final time of $20\,\mathrm{s}$ with $2000$ time steps, our solver completes the simulation in \(11.121\,\mathrm{s}\) using a single core of the workstation described above.

Figures~\ref{fig:37_aorta}--\ref{fig:37_gen3} compare the HDG pressure and flow
waveforms with the FVM reference solutions at the aortic root and at representative vessels of the first, second and third generations of bifurcations. The HDG scheme captures the pressure amplification and the progressive steepening of the wavefront toward the periphery, and reproduces the timing and magnitude of the reflected waves; the agreement with the in~vitro data is comparable to that of the FVM scheme, confirming that the HDG discretization of the coupled network, including the junction and terminal conditions, attains benchmark-level
accuracy.
\begin{figure}[!p]
  \centering
  \begin{subfigure}[b]{0.48\textwidth}
    \includegraphics[width=\linewidth]{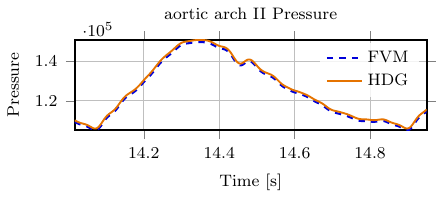}
    \caption{Aortic arch II -- pressure ($\mathrm{dyn/cm^2}$)}
  \end{subfigure}\hfill
  \begin{subfigure}[b]{0.48\textwidth}
    \includegraphics[width=\linewidth]{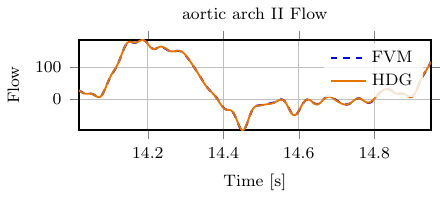}
    \caption{Aortic arch II -- flow ($\mathrm{ml/s}$)}
  \end{subfigure}
  \\[1ex]
  \begin{subfigure}[b]{0.48\textwidth}
    \includegraphics[width=\linewidth]{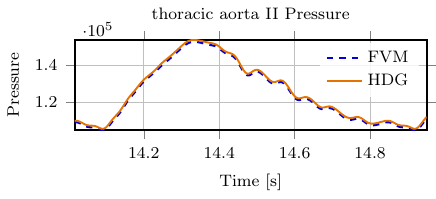}
    \caption{Thoracic aorta II -- pressure ($\mathrm{dyn/cm^2}$)}
  \end{subfigure}\hfill
  \begin{subfigure}[b]{0.48\textwidth}
    \includegraphics[width=\linewidth]{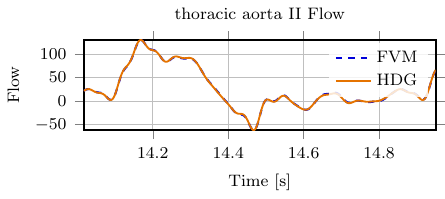}
    \caption{Thoracic aorta II -- flow ($\mathrm{ml/s}$)}
  \end{subfigure}
  \caption{Thirty-seven-artery network: pressure and flow waveforms at the
  midpoint of the aortic arch~II (top) and thoracic aorta~II (bottom). HDG
  (solid) compared with the finite-volume reference (FVM, dashed).}
  \label{fig:37_aorta}
\end{figure}
 
\begin{figure}[!p]
  \centering
  \begin{subfigure}[b]{0.48\textwidth}
    \includegraphics[width=\linewidth]{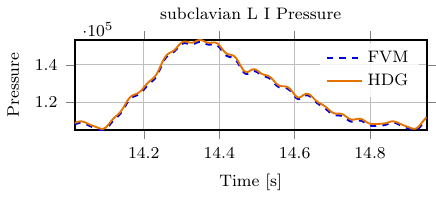}
    \caption{Left subclavian I -- pressure ($\mathrm{dyn/cm^2}$)}
  \end{subfigure}\hfill
  \begin{subfigure}[b]{0.48\textwidth}
    \includegraphics[width=\linewidth]{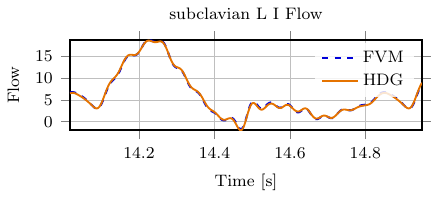}
    \caption{Left subclavian I -- flow ($\mathrm{ml/s}$)}
  \end{subfigure}
  \\[1ex]
  \begin{subfigure}[b]{0.48\textwidth}
    \includegraphics[width=\linewidth]{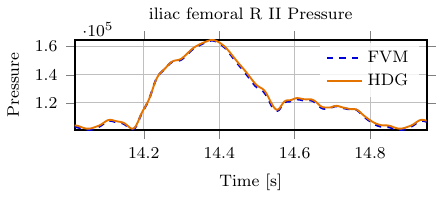}
    \caption{Right iliac-femoral II -- pressure ($\mathrm{dyn/cm^2}$)}
  \end{subfigure}\hfill
  \begin{subfigure}[b]{0.48\textwidth}
    \includegraphics[width=\linewidth]{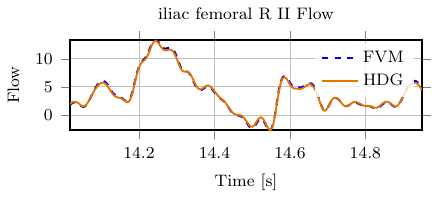}
    \caption{Right iliac-femoral II -- flow ($\mathrm{ml/s}$)}
  \end{subfigure}
  \caption{Thirty-seven-artery network, first generation of bifurcations:
  waveforms at the midpoint of the left subclavian~I (top) and right
  iliac-femoral~II (bottom). HDG (solid) versus FVM (dashed).}
  \label{fig:37_gen1}
\end{figure}
 
\begin{figure}[!p]
  \centering
  \begin{subfigure}[b]{0.48\textwidth}
    \includegraphics[width=\linewidth]{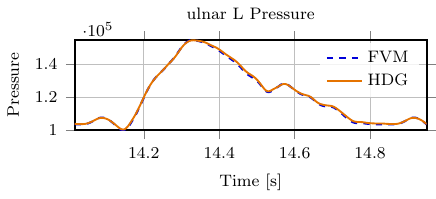}
    \caption{Left ulnar -- pressure ($\mathrm{dyn/cm^2}$)}
  \end{subfigure}\hfill
  \begin{subfigure}[b]{0.48\textwidth}
    \includegraphics[width=\linewidth]{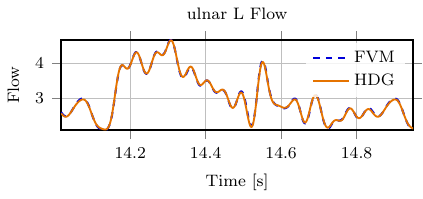}
    \caption{Left ulnar -- flow ($\mathrm{ml/s}$)}
  \end{subfigure}\\[1ex]
  \begin{subfigure}[b]{0.48\textwidth}
    \includegraphics[width=\linewidth]{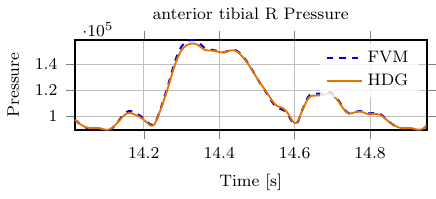}
    \caption{Right anterior tibial -- pressure ($\mathrm{dyn/cm^2}$)}
  \end{subfigure}\hfill
  \begin{subfigure}[b]{0.48\textwidth}
    \includegraphics[width=\linewidth]{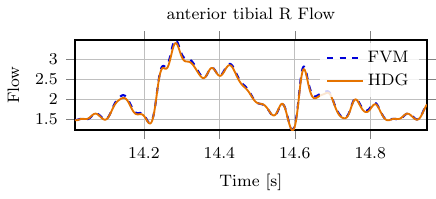}
    \caption{Right anterior tibial -- flow ($\mathrm{ml/s}$)}
  \end{subfigure}
  \caption{Thirty-seven-artery network, second generation of bifurcations: flow
  waveforms at the midpoint of the left ulnar (a) and right anterior tibial (b)
  arteries. HDG (solid) versus FVM (dashed).}
  \label{fig:37_gen2}
\end{figure}
 
\begin{figure}[!p]
  \centering
  \begin{subfigure}[b]{0.48\textwidth}
    \includegraphics[width=\linewidth]{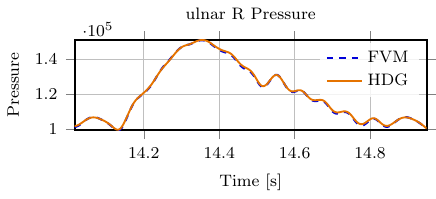}
    \caption{Right ulnar -- pressure ($\mathrm{dyn/cm^2}$)}
  \end{subfigure}\hfill
  \begin{subfigure}[b]{0.48\textwidth}
    \includegraphics[width=\linewidth]{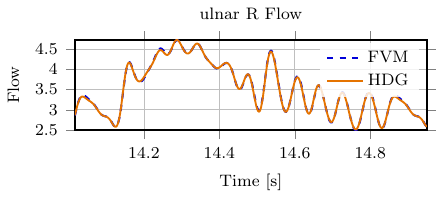}
    \caption{Right ulnar -- flow ($\mathrm{ml/s}$)}
  \end{subfigure}
  \\[1ex]
  \begin{subfigure}[b]{0.48\textwidth}
    \includegraphics[width=\linewidth]{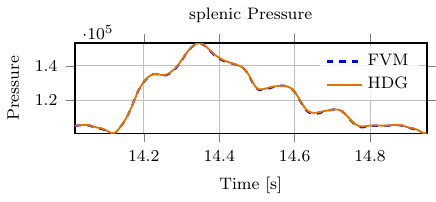}
    \caption{Splenic -- pressure ($\mathrm{dyn/cm^2}$)}
  \end{subfigure}\hfill
  \begin{subfigure}[b]{0.48\textwidth}
    \includegraphics[width=\linewidth]{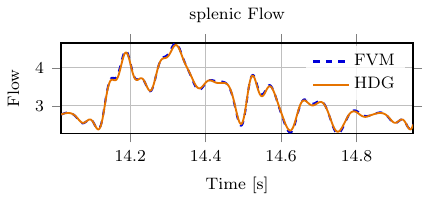}
    \caption{Splenic -- flow ($\mathrm{ml/s}$)}
  \end{subfigure}
  \caption{Thirty-seven-artery network, third generation of bifurcations:
  waveforms at the midpoint of the right ulnar (top) and splenic (bottom)
  arteries. HDG (solid) versus FVM (dashed).}
  \label{fig:37_gen3}
\end{figure}

\paragraph{Spatial mesh independence.}
To confirm that the reported waveforms are resolution independent, we repeat the $37$-artery simulation under global mesh refinement, from $2^{4}$ to $2^{7}$ cells per vessel, at fixed polynomial degree $p=2$ and maximum $\mathrm{BDF}$ order $2$. We monitor one representative artery from the aorta and from each generation of bifurcations: the thoracic aorta~II, the left subclavian~I (first generation), the right ulnar (second generation), and the splenic
(third generation). For each vessel, we track the systolic peak pressure, flow, and area across the refinement sequence (Figures~\ref{fig:37_aorta_mesh_ind}--\ref{fig:37_gen3_mesh_ind}). At every location, all three quantities remain within a $\pm1\%$ band of the finest-mesh value: the peak pressure and area are essentially converged already at the coarsest level, while the peak flow settles monotonically onto the finest-mesh value. All benchmark waveforms reported above are therefore computed on a mesh lying within this $\pm1\%$ band. The comparison with the reference schemes consequently measures the accuracy of the HDG scheme on a resolved discretization.
\begin{figure}[htbp]
  \centering
  \includegraphics[width=\linewidth]{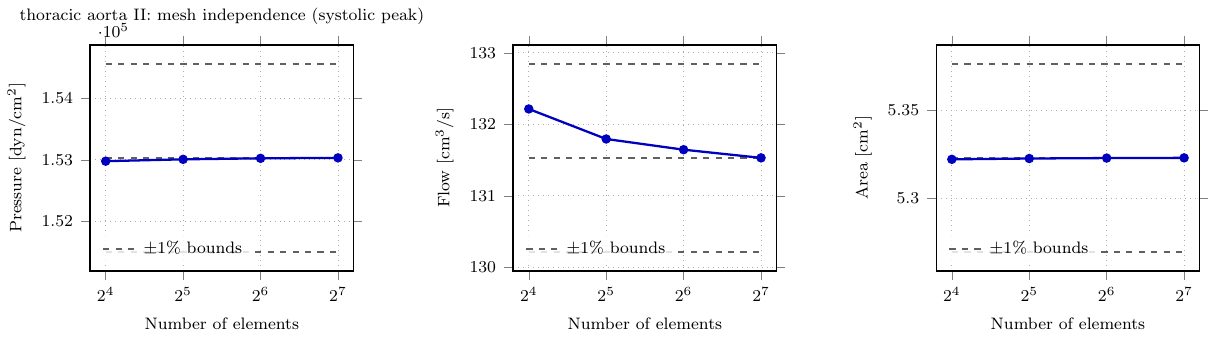}
  \caption{Mesh independence at the thoracic aorta~II (aorta): systolic peak
  pressure, flow and area under global refinement from $2^{4}$ to $2^{7}$
  elements per vessel ($p=2$). The dashed lines mark $\pm1\%$ of the
  finest-mesh value.}
  \label{fig:37_aorta_mesh_ind}
\end{figure}

\begin{figure}[htbp]
  \centering
  \includegraphics[width=\linewidth]{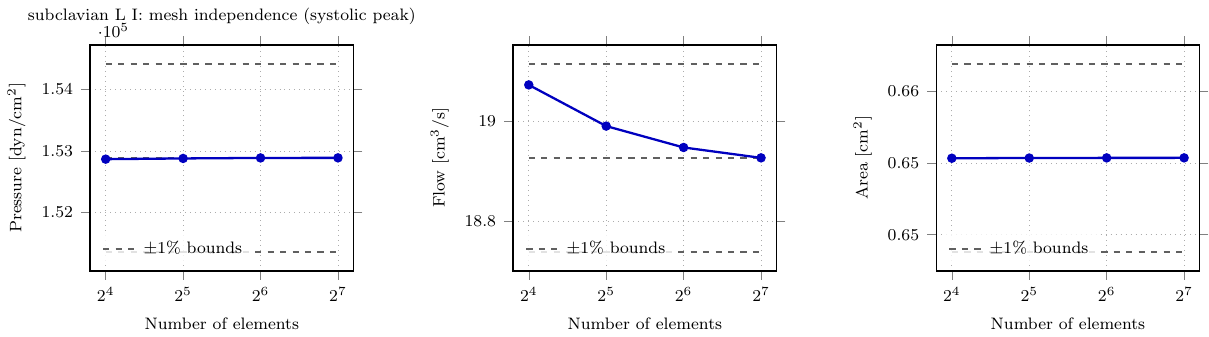}
  \caption{Mesh independence at the left subclavian~I (first generation).
  Details as in Figure~\ref{fig:37_aorta_mesh_ind}.}
  \label{fig:37_gen1_mesh_ind}
\end{figure}

\begin{figure}[htbp]
  \centering
  \includegraphics[width=\linewidth]{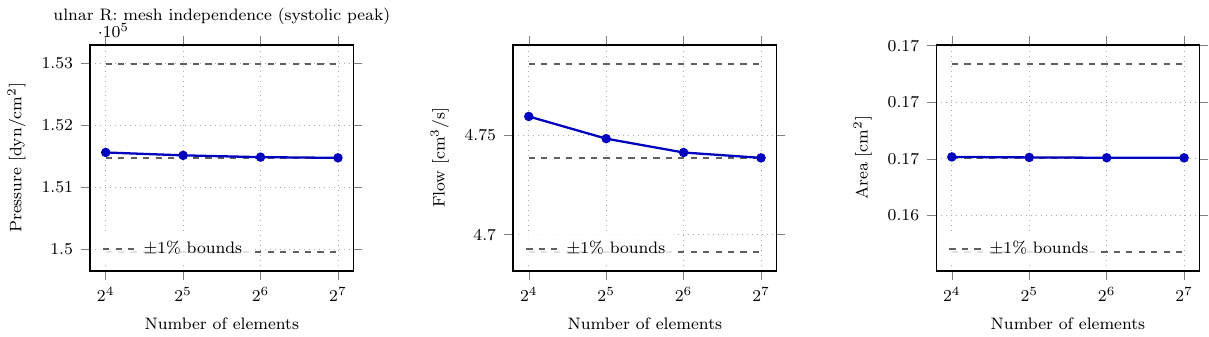}
  \caption{Mesh independence at the right ulnar (second generation).
  Details as in Figure~\ref{fig:37_aorta_mesh_ind}.}
  \label{fig:37_gen2_mesh_ind}
\end{figure}

\begin{figure}[htbp]
  \centering
  \includegraphics[width=\linewidth]{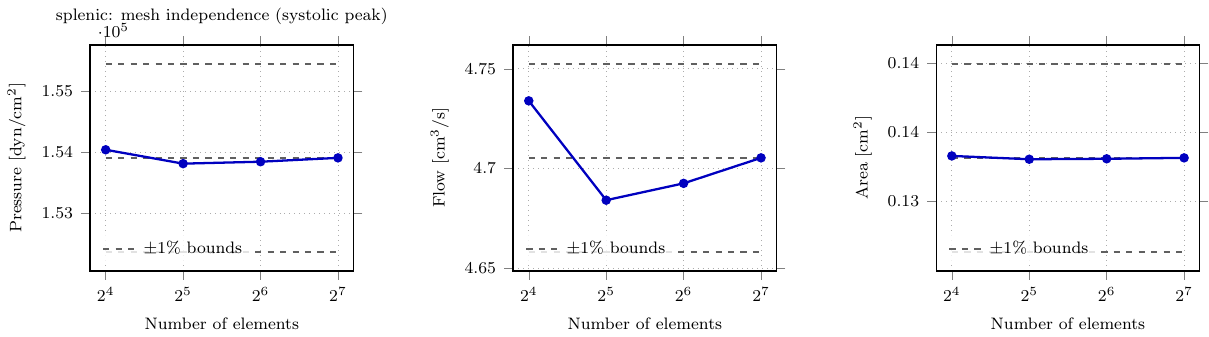}
  \caption{Mesh independence at the splenic artery (third generation).
  Details as in Figure~\ref{fig:37_aorta_mesh_ind}.}
  \label{fig:37_gen3_mesh_ind}
\end{figure}

\paragraph{Time resolution.}
We integrate the semi-discrete system using the implicit, variable-step, variable-order $\mathrm{BDF}$ methods described above. The time-resolution experiments use the same backend and residual formulation. To study the effect of the maximum $\mathrm{BDF}$ order on the computational cost, we fix the spatial polynomial degree
to $p=3$ and consider two levels of global mesh refinement. The
simulations are performed on a single processor with fixed absolute tolerance set to $10^{-7}$, and relative error tolerance set to $10^{-6}$. The resulting wall-clock times and time-integration statistics are reported in Figure~\ref{fig:time_resolution} and Table~\ref{tab:bdf_statistics}.

For both mesh resolutions, the smallest computational cost is obtained when the maximum $\mathrm{BDF}$ order is set to two. This behaviour is closely related to the number of accepted time steps: restricting the method to $\mathrm{BDF}1$ requires substantially more time steps, while allowing $\mathrm{BDF}2$
significantly reduces the number of steps needed to reach the final time. Increasing the maximum order beyond two does not result in a further reduction in computational cost; see Figure~\ref{fig:time_resolution}. In particular, for the higher mesh refinement, the number of accepted steps increases noticeably for the higher maximum orders, accompanied by an increase in the total number of nonlinear iterations; see Table~\ref{tab:bdf_statistics}.

The results also demonstrate that the maximum $\mathrm{BDF}$ order is only an upper bound for the adaptive time integrator. The selected order may vary during the simulation depending on the estimated local truncation error and the solution history. For example, when higher maximum orders are permitted, the dominant order need not coincide with the maximum allowed order, as indicated by the statistics in Table~\ref{tab:bdf_statistics}. Thus, the observed computational cost is determined by the combined adaptive selection of the time step and $\mathrm{BDF}$ order.

The stability properties of the $\mathrm{BDF}$ family provide useful context for
this behaviour. $\mathrm{BDF}$1 and $\mathrm{BDF}$2 are $A$-stable, whereas the higher-order
$\mathrm{BDF}$ methods are $A(\alpha)$-stable \cite{ascher1998computer}. The observed cost reflects the combined adaptive selection of step size and order. The experiments show that, for this problem and the chosen
error tolerances, allowing a maximum $\mathrm{BDF}$ order of two gives the lowest
observed computational cost.
\begin{figure}
  \centering
  \begin{subfigure}[b]{0.49\textwidth}
    \includegraphics[width=\linewidth]{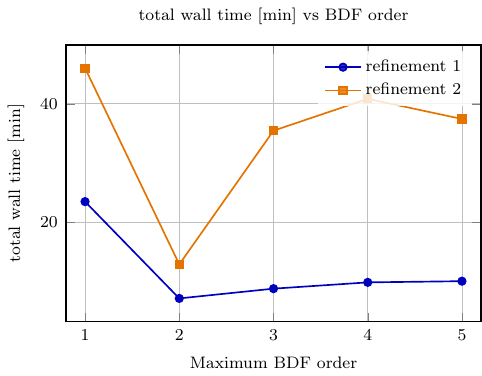}
    \caption{Total wall-clock time}
  \end{subfigure}\hfill
  \begin{subfigure}[b]{0.47\textwidth}
    \includegraphics[width=\linewidth]{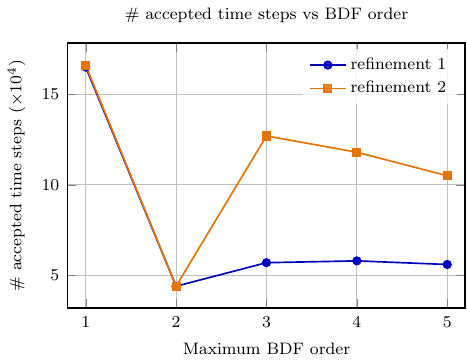}
    \caption{Number of accepted time steps}
  \end{subfigure}
  \caption{Cost of a single $20\,\mathrm{s}$ simulation of the $37$-vessel network on one processor, as a function of the maximum $\mathrm{BDF}$ order,
  with HDG degree $p = 3$. In every case the second-order scheme is the cheapest and first order the most expensive.}
  \label{fig:time_resolution}
\end{figure}

\begin{table}[ht]
\centering
\caption{BDF time-integration statistics for a single $20\,\mathrm{s}$
simulation of the $37$-vessel network on one processor, with HDG degree
$p = 3$, as a function of the maximum $\mathrm{BDF}$ order $q_{\max}$:
number of error-test failures $N_{\mathrm{fail}}$, number of nonlinear
iterations $N_{\mathrm{NI}}$, and the dominant order $q_{\mathrm{dom}}$
(the order at which the most steps were accepted), for one and two levels of
global mesh refinement.}
\label{tab:bdf_statistics}
\begin{tabular}{c|rrr|rrr}
\toprule
& \multicolumn{3}{c|}{One refinement}
& \multicolumn{3}{c}{Two refinements} \\
$q_{\max}$
& $N_{\mathrm{fail}}$ & $N_{\mathrm{NI}}$ & $q_{\mathrm{dom}}$
& $N_{\mathrm{fail}}$ & $N_{\mathrm{NI}}$ & $q_{\mathrm{dom}}$ \\
\midrule
1 & 2  & 174763 & 1 & 2  & 175862 & 1 \\
2 & 2  &  51561 & 2 & 2  &  47861 & 2 \\
3 & 19 &  64176 & 3 & 10 & 134431 & 3 \\
4 & 23 &  71892 & 4 & 12 & 150778 & 4 \\
5 & 35 &  71517 & 3 & 25 & 136992 & 3 \\
\bottomrule
\end{tabular}
\end{table}
\subsection{Study for Large Arterial Networks}
\subsubsection{7999 vessels}\label{sec:liver}
To verify the solver on a large-scale network with liver-like topology, we use
a synthetic hepatic arterial tree of $7{,}999$ vessels, whose topology and
branch geometry were generated with the open-source constrained constructive
optimisation (CCO) tool \textsc{OpenCCO} \cite{kerautret2023opencco}. The CCO
algorithm grows a space-filling binary tree by iteratively inserting terminal
segments that minimise a global cost functional (total intravascular volume).
Vessel radii are then assigned by Murray's law \cite{murray1926physiological}:
starting from the root radius $r_0$, both daughters of each bifurcation receive
$r_{\text{parent}}/2^{1/3}$, with a lower bound $r_0/\kappa$. All quantities
are given in CGS units. The vessel wall follows the
radius law~\eqref{eq:ls_wall}, and every terminal is closed with a single
resistor $R_2=\Delta P/Q_{\text{term}}$, where $\Delta P=60\,\mathrm{mmHg}$ is
the terminal pressure of the CCO construction and $Q_{\text{term}}$ is the
uniform terminal flow. The mean inflow $\bar Q_{\text{in}}$ is chosen such that
steady Poiseuille flow through the network gives the CCO perfusion pressure of
$100\,\mathrm{mmHg}$ at the root; the inlet is driven by a synthetic periodic
inflow of period $1\,\mathrm{s}$ with this mean. The parameters are listed in
Table~\ref{tab:params}, and the pressure field at peak inflow in the second
cycle, $t=1.213\,\mathrm{s}$, is shown in Figure~\ref{fig:8k_results}.

\begin{table}[]
\centering
\caption{Simulation parameters for the liver vascular tree (CGS units).}
\label{tab:params}
\begin{tabular}{@{}lll@{}}
\toprule
Symbol & Description & Value \\
\midrule
\multicolumn{3}{@{}l}{\textit{Network geometry}}\\
$N_\mathrm{v}$ & Vessels (cells) & 7999 \\
$N_\mathrm{t}$ & Terminals & 4000 \\
$N_\mathrm{j}$ & Junctions & 3999 \\
$g_\mathrm{max}$ & Max.\ generation depth & 57 \\
$r_0$ & Root radius & $0.174$ cm \\
$\kappa$ & Radius-floor ratio & 8 \\
$r_d$ & Radius range & $0.022$--$0.174$ cm \\
$\ell$ & Length range & $5.2\times10^{-3}$--$1.00$ cm \\
$A_d$ & Cross-section range & $1.49\times10^{-3}$--$9.55\times10^{-2}$ cm$^2$ \\
\midrule
\multicolumn{3}{@{}l}{\textit{Blood properties}}\\
$\rho$ & Density & $1.06$ g\,cm$^{-3}$ \\
$\mu$ & Dynamic viscosity & $0.036$ P \\
$\xi$ & Friction constant & $2$ \\
$m$ & Tube-law exponent & $0.5$ \\
\midrule
\multicolumn{3}{@{}l}{\textit{Vessel wall, Eq.~\eqref{eq:ls_wall}}}\\
$E$ & Young's modulus & $4.0\times10^{6}$ dyn\,cm$^{-2}$ \\
$k_1,\,k_2,\,k_3$ & Wall-law coefficients & $6.54\times10^{5}$ dyn\,cm$^{-2}$, $-4.62$ cm$^{-1}$, $2.63\times10^{5}$ dyn\,cm$^{-2}$ \\
$h_w/r_d$ & Relative wall thickness & $0.14$--$0.21$ \\
$\beta$ & $\tfrac{4}{3}\sqrt{\pi}Eh_w$ & $4.4\times10^{4}$--$2.3\times10^{5}$ dyn\,cm$^{-1}$ \\
$P_0,\,P_d$ & Reference / diastolic pressure & $0$ \\
\midrule
\multicolumn{3}{@{}l}{\textit{Inflow and terminal BC (single resistor)}}\\
$\bar Q_\mathrm{in}$ & Mean inflow & $1.04$ cm$^3$\,s$^{-1}$ \\
$\Delta P$ & Terminal pressure & $7.98\times10^{4}$ dyn\,cm$^{-2}$ (60 mmHg) \\
$R_2$ & Terminal resistance & $3.07\times10^{8}$ dyn\,s\,cm$^{-5}$ \\
$R_1,\,C,\,P_\mathrm{out}$ & (resistive limit) & $0$ \\
\midrule
\multicolumn{3}{@{}l}{\textit{Discretization \& solver}}\\
$p$ & DG degree & $2$ \\
-- & Global refinements & $1$ \\
-- & Numerical flux & HLL \\
-- & Max.\ BDF order & $5$ \\
atol,\,rtol & Tolerances & $10^{-7},\,10^{-4}$ \\
$T$ & Final time & $2$ s \\
\bottomrule
\end{tabular}
\end{table}

\begin{figure}[H]
  \centering
  \begin{subfigure}[c]{0.41\linewidth}
    \centering
    \includegraphics[width=\linewidth]{
 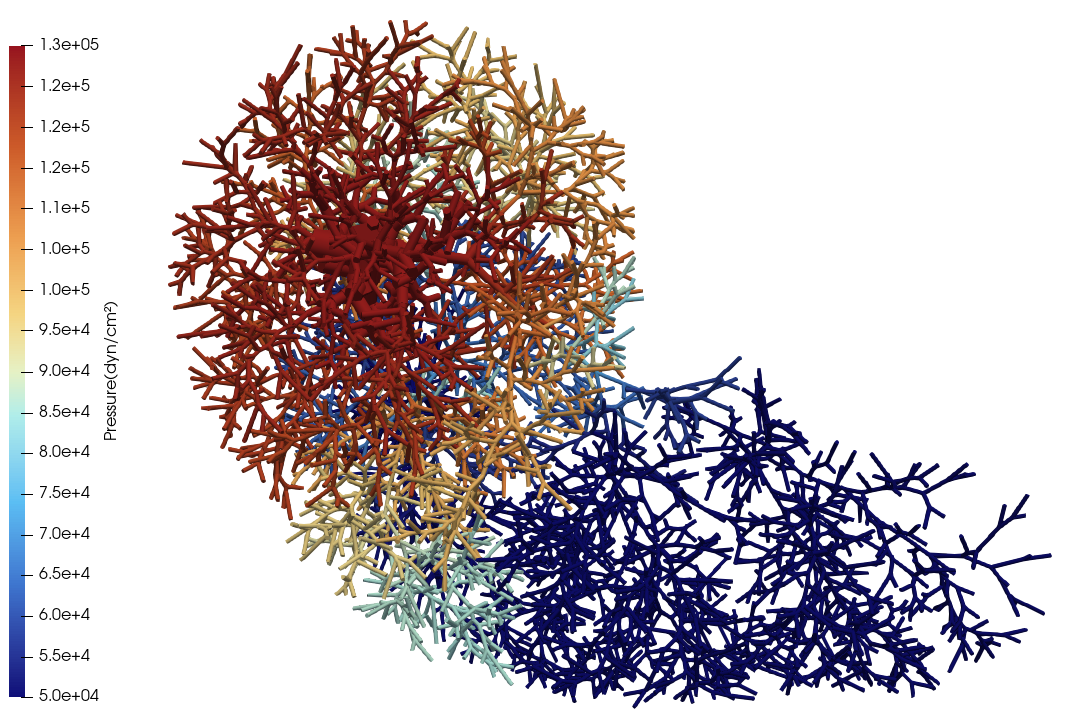}
    \caption{Pressure field at peak inflow, $t = 1.213\,\mathrm{s}$.}
    \label{fig:8k_results}
  \end{subfigure}\hfill
  \begin{subfigure}[c]{0.54\linewidth}
    \centering
    \includegraphics[width=\linewidth]{
      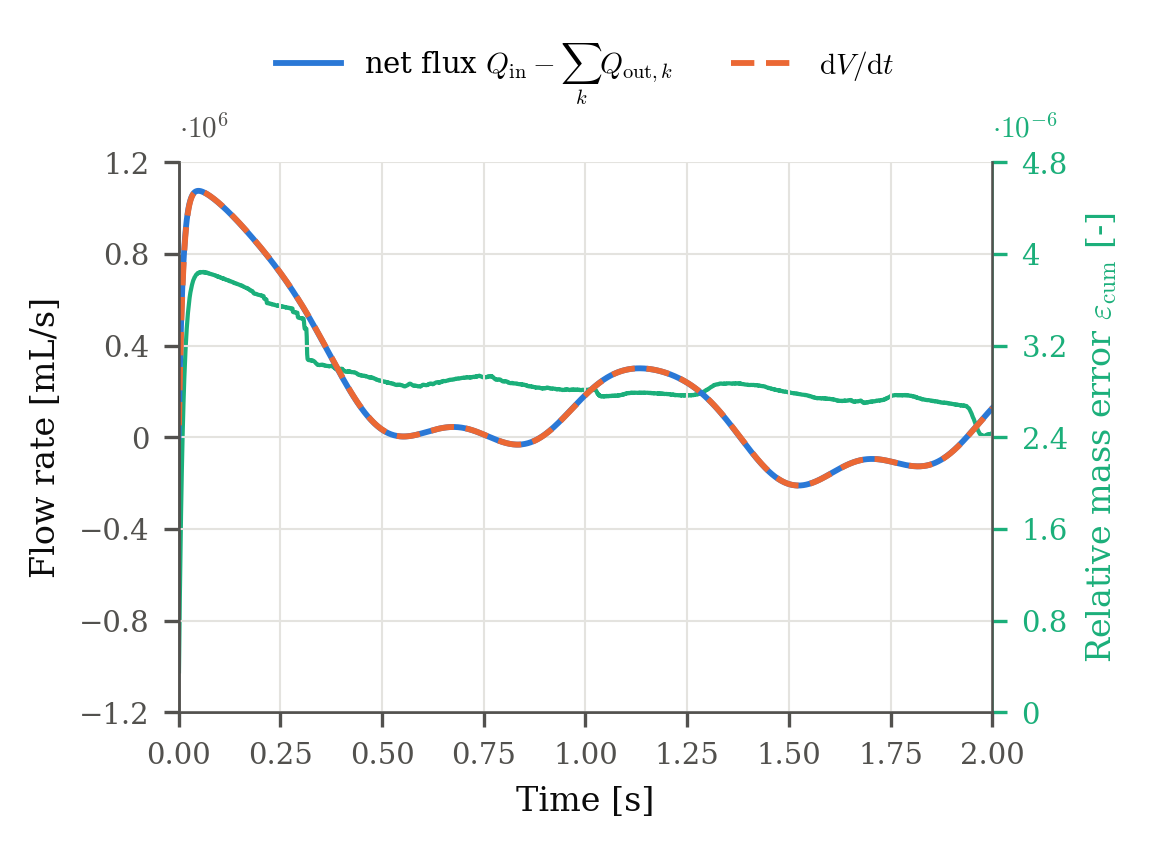}
    \caption{Global mass conservation.}
    \label{fig:8k_mass}
  \end{subfigure}
  \caption{Liver network with $7{,}999$ vessels, simulated up to
  $T = 2\,\mathrm{s}$ with polynomial degree $p = 2$, two global mesh
  refinements and BDF of maximum order five.
  (a)~Computed pressure field at peak inflow in the second cycle.
  (b)~Net boundary flux $Q_\mathrm{in} - \sum_k Q_{\mathrm{out},k}$ and the
  rate of change of the total network volume $\mathrm{d}V/\mathrm{d}t$
  (left axis), and the cumulative relative mass error
  $\varepsilon_\mathrm{cum}(t) = \bigl|V(t) - V(0) - \int_0^t
  (Q_\mathrm{in} - \sum_k Q_{\mathrm{out},k})\,\mathrm{d}s\bigr| / V(t)$
  (right axis). $\mathrm{d}V/\mathrm{d}t$ is computed by finite differences
  between consecutive outputs. The two flow curves coincide, and
  $\varepsilon_\mathrm{cum}$ stays of order $10^{-6}$.}
  \label{fig:8k}
\end{figure}
\subsubsection{Cerebral arteriolar network}
\label{sec:hc}
As a large, physiologically motivated application, we simulate the arteriolar
bed of the posterior parietal branch (PPB) of the middle cerebral artery (MCA)
from Blanco et al.~\cite{blanco2017blood}. The bed was generated with the
constrained constructive optimisation (CCO) technique and is a binary tree with
19,999 vessels, $9{,}999$ bifurcations and $10{,}000$ terminals. For every segment,
\cite{blanco2017blood} provides the mean area and mean pressure; we use them as
the reference state of the tube law, i.e.\ $A_d$ is the mean area and $P_d$ the
mean pressure of the segment, and $\beta$ is fitted to the pressure--area data
of the tree. The simulation starts from this mean state at zero velocity. As in the
reference study, the terminals are purely resistive, with
$P_\mathrm{out} = 10\,\mathrm{mmHg}$, and their resistances are set from the
mean terminal pressures and flows. Since the tree is driven there by an inlet
pressure, whereas our model prescribes the inflow, we convert the inlet
pressure waveform of~\cite{blanco2017blood} into an inflow $Q_\mathrm{in}(t)$
with a two-element Windkessel model of the tree, and scale it so that its mean
equals the total flow leaving the terminals,
$1.715\times10^{-7}\,\mathrm{m^3\,s^{-1}}$. Because the computation starts from
the mean state, the transient is short, and two cardiac cycles of
$1\,\mathrm{s}$ are sufficient.

Figure~\ref{fig:cer_map} shows the mean pressure over the second cycle, $t \in [1, 2]\,\mathrm{s}$, which decreases from the feeding segment towards the peripheral arterioles. Since individual arterioles cannot be resolved in the
figure, we average the mean pressure over all vessels of a given diameter, as in~\cite{blanco2017blood}: it is $58.4\,\mathrm{mmHg}$ for diameters
$190$--$210\,\mu\mathrm{m}$ and $51.0\,\mathrm{mmHg}$ for
$30$--$50\,\mu\mathrm{m}$, compared with $56.5$ and $50.7\,\mathrm{mmHg}$ reported there. This agreement confirms that the network and boundary data are set up
correctly. Figure~\ref{fig:cer_ida} summarises the adaptive BDF time integration performed by IDA. Over the two cycles, IDA takes $1813$ time steps with $2407$
Newton iterations ($1.33$ per step), $44$ Jacobian setups and $9$ error-test
failures, and no nonlinear convergence failures. About half of the steps are
spent in the start-up transient ($t \le 0.03\,\mathrm{s}$), where the step size
grows from $10^{-8}$ to $10^{-3}\,\mathrm{s}$ and the order rises from one to
five. Afterwards the step size settles at about $2\,\mathrm{ms}$, the order
varies between three and five, and both cardiac cycles require a comparable
number of steps.

\begin{figure}[H]
  \centering
  \begin{subfigure}[c]{0.41\linewidth}
    \centering
    \includegraphics[width=\linewidth]{
      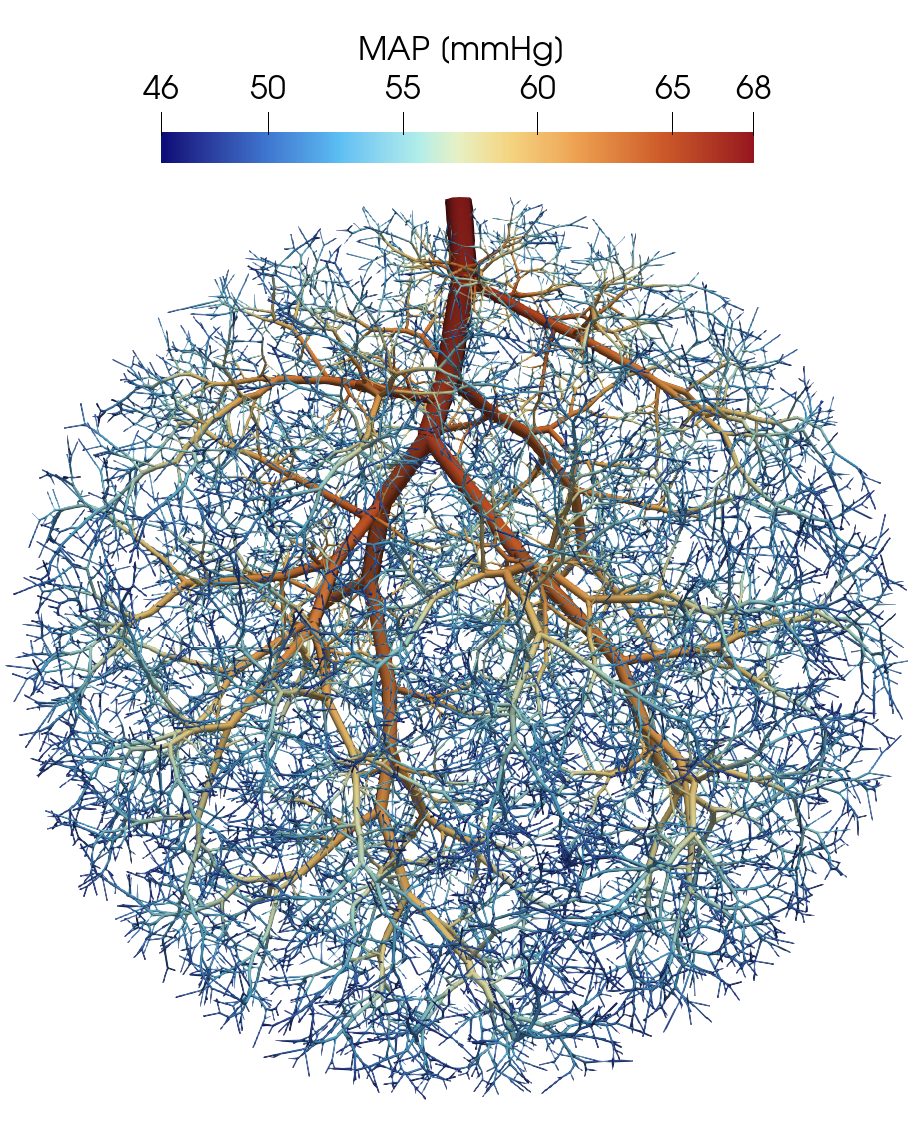}
    \caption{Mean arterial pressure (MAP) over $t \in [1, 2]\,\mathrm{s}$.}
    \label{fig:cer_map}
  \end{subfigure}\hfill
  \begin{subfigure}[c]{0.54\linewidth}
    \centering
    \includegraphics[width=\linewidth]{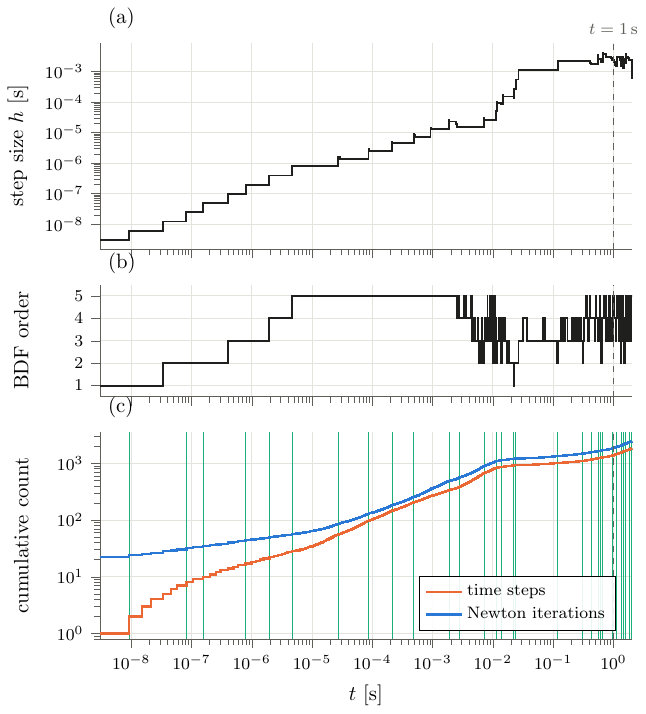}
    \caption{IDA statistics.}
    \label{fig:cer_ida}
  \end{subfigure}
  \caption{Arteriolar bed of the posterior parietal branch of the middle
  cerebral artery ($19{,}999$ vessels, normotensive geometry
  of~\cite{blanco2017blood}).
  (a)~Pressure averaged over the second cardiac cycle.
  (b)~Adaptive BDF time integration on a logarithmic time axis: step size,
  BDF order, and cumulative counts of time steps and Newton iterations.
  Thin green vertical lines mark steps where the Jacobian is updated. During the other steps the algorithm reuses the previously computed Jacobian matrix.
  The average is $1.33$ Newton iterations per step. The computation takes 1,813 time steps and 2,407 Newton iterations.}
  \label{fig:cerebral}
\end{figure}
\subsubsection{Synthetic large-scale networks: 50k and 100k vessels}
\label{sec:large_scale}
In this section, we go beyond the liver network described in
Section~\ref{sec:liver}. We generate two synthetic vascular trees of
$50{,}000$ and $100{,}000$ vessels using \textsc{svVascularize}
\cite{sexton2025svvascularize}, which constructs space-filling binary trees by CCO.
For each of the $N$ segments, the generator returns the segment endpoints, the length $L_i$, the reference
(diastolic) radius $r_{d,i}$ and the terminal (design) volumetric flow $Q_i$;
the simulations use the generated tree morphology, connectivity, lengths, radii and design flows directly. Both networks are run to a
final time $T = 1\,\mathrm{s}$ with the direct solver \textsc{mumps}. For the two 50k and 100k trees, we apply $p=2$ with $1$ cell per vessel. In each
case MPI process counts of $4$, $6$, $8$, $10$ and $12$ are tested.

\paragraph{Vessel wall and constitutive parameters.}
All quantities in this section are given in CGS units, the unit system in which
\textsc{svVascularize} generates the networks. The segments produced by
\textsc{svVascularize} have constant cross-section, so that
$r_{\text{in}}=r_{\text{out}}=r_{d}$ and the reference cross-section is constant
along each vessel,
\begin{equation}
  A_{d,i} \;=\; A_{0,i} \;=\; \pi\,r_{d,i}^{2}.
  \label{eq:ls_area}
\end{equation}
Although these networks are synthetic rather than hepatic, they are closed with
the same constitutive model as the liver network of Section~\ref{sec:liver},
so that both test cases differ only in geometry. The tube law is
\begin{equation}
  P(A,x)-P_{0} \;=\; P_{d} \;+\; \frac{\beta(x)}{A_{d}(x)}
      \Big(A^{m}-A_{d}(x)^{m}\Big),
  \qquad
  \beta(x)=\tfrac{4}{3}\sqrt{\pi}\,E\,h_{w}(x),
  \quad m=\tfrac12 ,
  \label{eq:ls_tubelaw}
\end{equation}
with external and diastolic reference pressures $P_{0}=P_{d}=0$. The wall
stiffness entering $\beta$ follows the empirical radius law of
\cite[Sec. 5.2.2]{MR5106706},
\begin{equation}
  \frac{E\,h_{w}(r_{d})}{r_{d}} \;=\; k_{1}\,e^{k_{2}\,r_{d}} + k_{3},
  \qquad\Longrightarrow\qquad
  h_{w}(r_{d}) \;=\; \frac{r_{d}}{E}\Big(k_{1}\,e^{k_{2}\,r_{d}}+k_{3}\Big),
  \label{eq:ls_wall}
\end{equation}
so that
$\beta_i=\tfrac{4}{3}\sqrt{\pi}\,r_{d,i}\big(k_{1}e^{k_{2}r_{d,i}}+k_{3}\big)$
depends on the local radius alone. The coefficients are
$k_{1}=6.5383\times10^{5}\,\mathrm{dyn\,cm^{-2}}$,
$k_{2}=-4.6154\,\mathrm{cm^{-1}}$ and
$k_{3}=2.6326\times10^{5}\,\mathrm{dyn\,cm^{-2}}$. Since only the product
$E h_w$ enters~\eqref{eq:ls_tubelaw}, the modulus $E$ affects the reported wall
thickness but not the dynamics; we use
$E=4.0\times10^{6}\,\mathrm{dyn\,cm^{-2}}$ ($0.4\,\mathrm{MPa}$), which gives
$h_w/r_d\approx0.18$--$0.23$. The resulting $\beta_i$ and $h_{w,i}$ are stored
per vessel, together with $A_{d,i}$, $L_{i}$ and $r_{d,i}$. Blood density and
viscosity are those used by \textsc{svVascularize} to build the trees,
$\rho=1.06\,\mathrm{g\,cm^{-3}}$ and $\mu=0.036\,\mathrm{P}$.

\textit{Terminal boundary parameters.}
Every terminal vessel is closed with the Windkessel outlet of
Section~\ref{sec:liver}, so that at each outlet the trace satisfies
$P(\widehat A)=R_{1}\widehat A\,\widehat U+P_{c}$ together with the capacitor
equation $C\dot P_{c}=\widehat A\,\widehat U-(P_{c}-P_{\text{out}})/R_{2}$. The
total peripheral resistance of each terminal is set from the design flow
$Q_i$ returned by the generator,
\begin{equation}
  R^{i}_{\text{tot}} \;=\; \frac{\Delta P}{Q_{i}},
  \label{eq:ls_Rtot}
\end{equation}
where $\Delta P=7.9993\times10^{4}\,\mathrm{dyn\,cm^{-2}}$ ($60\,\mathrm{mmHg}$)
is the terminal pressure prescribed in the tree construction. With the uniform
terminal flow $Q_i=6.25\times10^{-5}\,\mathrm{cm^3\,s^{-1}}$ of both networks,
this gives $R^{i}_{\text{tot}}=1.28\times10^{9}\,\mathrm{dyn\,s\,cm^{-5}}$ at
every outlet. In general the total resistance is split between the proximal
and distal resistors, and the compliance is fixed through an $RC$ time
constant $\tau$,
\begin{equation}
  R_{1}^{i}=\gamma\,R^{i}_{\text{tot}},\qquad
  R_{2}^{i}=(1-\gamma)\,R^{i}_{\text{tot}},\qquad
  C^{i}=\frac{\tau}{R^{i}_{\text{tot}}}.
  \label{eq:ls_rcr}
\end{equation}
The scaling runs reported below use the pure-resistance terminal, recovered as
the limit $R_{1}\to0$, $C\to0$ of~\eqref{eq:ls_rcr}, i.e.\ $R_{1}^{i}=0$,
$R_{2}^{i}=R^{i}_{\text{tot}}$, $C^{i}=0$ and $P_{\text{out}}=0$, which is the
single-resistance special case of the outlet operator in
Section~\ref{sec:liver}. The single arterial inlet ($\mathrm{bid}=0$) is driven
by a prescribed inflow $Q_{\text{in}}(t)$ whose mean equals the design root
flow of each tree ($1.5625$ and $3.125\,\mathrm{cm^3\,s^{-1}}$ for the 50k and
100k networks). The remaining nodes are tagged as terminals
($\mathrm{bid}=1,\dots,n_{\text{out}}$, one per outlet) or as interior
junctions, and only terminal nodes carry nonzero resistances. In summary, all
geometric and flow quantities $\{L_i,\,r_{d,i},\,A_{d,i},\,Q_i\}$ are taken
directly from \textsc{svVascularize}, while the material constants
$(E,k_{1},k_{2},k_{3})$, the terminal pressure $\Delta P$ and the reference
pressures $(P_{0},P_{d},P_{\text{out}})$ are the only modelling inputs; these
are collected in Table~\ref{tab:ls_params}.

\begin{table}[H]
  \centering
  \caption{Parameters of the synthetic large-scale networks (CGS units).
  Geometry and flow are produced by \textsc{svVascularize}; the derived areas,
  wall thickness~\eqref{eq:ls_wall} and terminal resistances
  \eqref{eq:ls_Rtot}--\eqref{eq:ls_rcr} follow from the constants listed.}
  \begin{tabular}{llll}
    \hline
    Symbol & Meaning & Value & Source \\
    \hline
    $L_i,\ r_{d,i},\ Q_i$ & length, radius, terminal flow & per vessel & \textsc{svVascularize} \\
    $A_{0,i}=A_{d,i}$ & reference area & $\pi r_{d,i}^2$ & Eq.~\eqref{eq:ls_area} \\
    $\rho$            & blood density                   & $1.06\,\mathrm{g\,cm^{-3}}$ & \textsc{svVascularize} \\
    $\mu$             & blood viscosity                 & $0.036\,\mathrm{P}$ & \textsc{svVascularize} \\
    $E$               & wall Young's modulus            & $4.0\times10^{6}\,\mathrm{dyn\,cm^{-2}}$ & assumed \\
    $k_1$             & wall-law coefficient            & $6.5383\times10^{5}\,\mathrm{dyn\,cm^{-2}}$ & assumed \\
    $k_2$             & wall-law coefficient            & $-4.6154\,\mathrm{cm^{-1}}$ & assumed \\
    $k_3$             & wall-law coefficient            & $2.6326\times10^{5}\,\mathrm{dyn\,cm^{-2}}$ & assumed \\
    $m$               & tube-law exponent               & $1/2$ & assumed \\
    $P_0,\ P_d$       & external, diastolic pressure    & $0$ & assumed \\
    $\Delta P$        & terminal pressure               & $7.9993\times10^{4}\,\mathrm{dyn\,cm^{-2}}$ (60 mmHg) & \textsc{svVascularize} \\
    $R^{i}_{\text{tot}}$ & terminal resistance          & $1.28\times10^{9}\,\mathrm{dyn\,s\,cm^{-5}}$ & Eq.~\eqref{eq:ls_Rtot} \\
    $P_{\text{out}}$  & outlet pressure                 & $0$ & assumed \\
    \hline
  \end{tabular}
  \label{tab:ls_params}
\end{table}
\paragraph{Flow solution.}
Figure~\ref{fig:50k_100k_pressure} shows the pressure distributions at
$t=0.213\,\mathrm{s}$ for the $50\mathrm{k}$ and $100\mathrm{k}$ networks. The networks are partitioned
over at most $12$ MPI subdomains. The pressure decreases monotonically from
the single arterial inlet toward the terminal outlets, qualitatively consistent with the smaller liver network of
Section~\ref{sec:liver} and confirming that the junction and terminal conditions of \eqref{eq:ls_Rtot}--\eqref{eq:ls_rcr} are enforced correctly at this scale. The spatial pressure gradient and the branching structure are well resolved in both cases.

\textit{Strong scaling.}
Figure~\ref{fig:all_scaling} reports wall times for the dominant kernels as a
function of MPI process count on a log--log scale, with an ideal-scaling line
included for reference. The residual and Jacobian assembly kernels account for
most of the computational work and scale nearly ideally up to $8$ processes
for both network sizes. The linear solver limits the overall strong scaling:
each direct solve is inexpensive at these problem sizes, leaving too little
work per process to offset the overhead of parallel factorization. Although the
individual solves are cheap, their repeated execution over many time steps
produces a cumulative cost that remains approximately constant as the process
count increases.

Beyond $8$ processes, the limited scalability of the direct solves leads to
diminishing returns from additional parallelism. Similar behaviour is observed
for the $8\mathrm{k}$ liver network. The assembly kernels remain the dominant
cost and continue to scale well, suggesting that communication associated with
the HDG skeleton coupling and terminal conditions remains manageable over the
network sizes considered. Larger systems may improve parallel efficiency in
the solve phase by increasing the work available per process.
\begin{figure}[H]
    \centering   
    \begin{subfigure}[t]{0.50\linewidth}
        \centering
        \includegraphics[width=\linewidth]{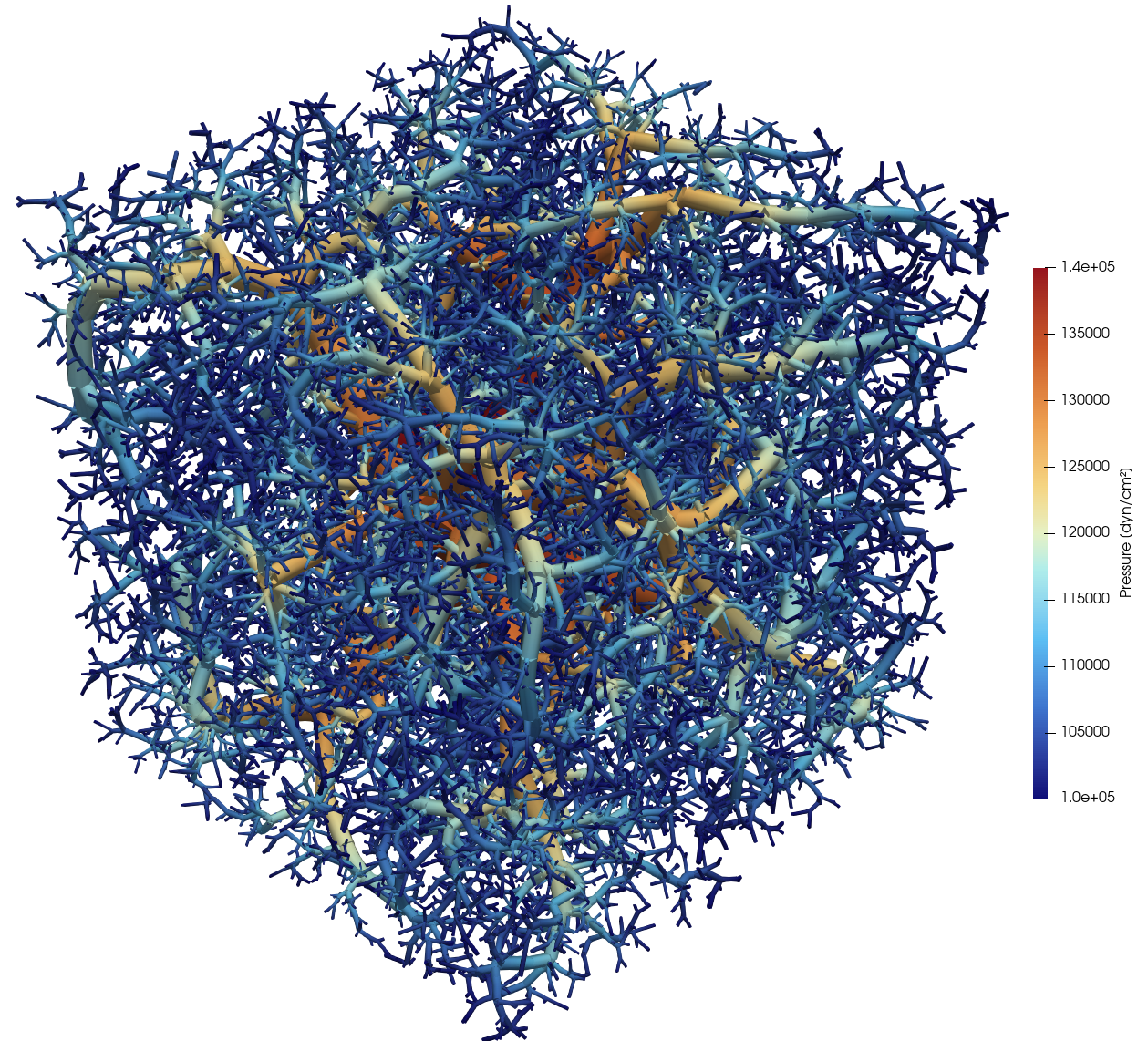}
        \caption{$50 \mathrm{k}$ vessel networks.}
        \label{fig:50k_network_pressure}
    \end{subfigure}
    \hfill
    \begin{subfigure}[t]{0.46\linewidth}
        \centering
        \includegraphics[width=\linewidth]{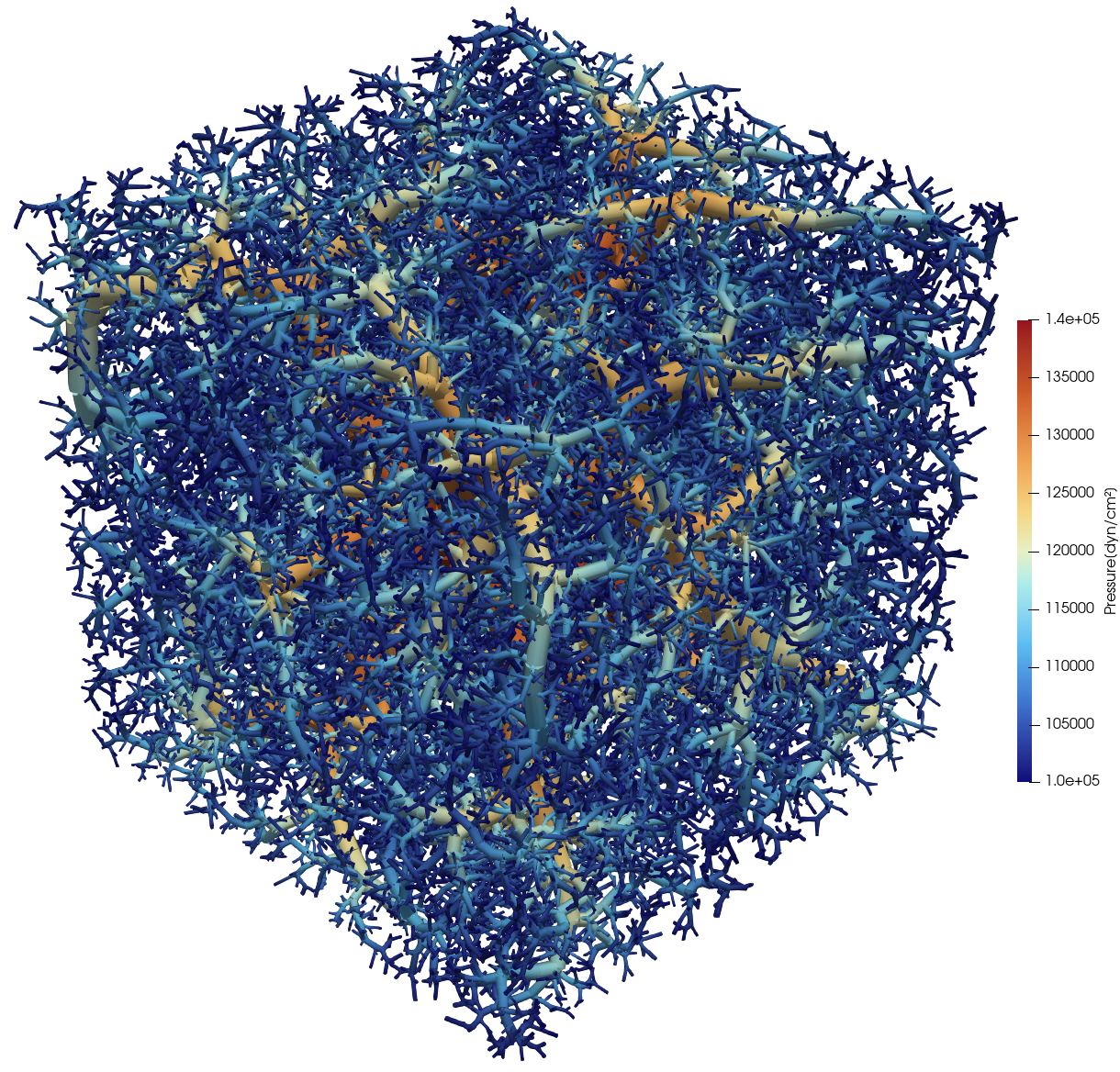}
        \caption{$100 \mathrm{k}$ vessel networks.}
        \label{fig:100k_network_pressure}
    \end{subfigure}
    \caption{Pressure distribution at $t=0.213\,\mathrm{s}$, when the prescribed inlet flow
$Q_{\text{in}}(t)$ reaches its maximum ($1.0$--$1.4\times10^{5}\,\mathrm{dyn\,cm^{-2}}$,
about $75$--$105\,\mathrm{mmHg}$).}
    \label{fig:50k_100k_pressure}
\end{figure}

\begin{figure}[H]
  \centering
  \begin{subfigure}[t]{0.48\linewidth}
    \centering
    \includegraphics[width=\linewidth]{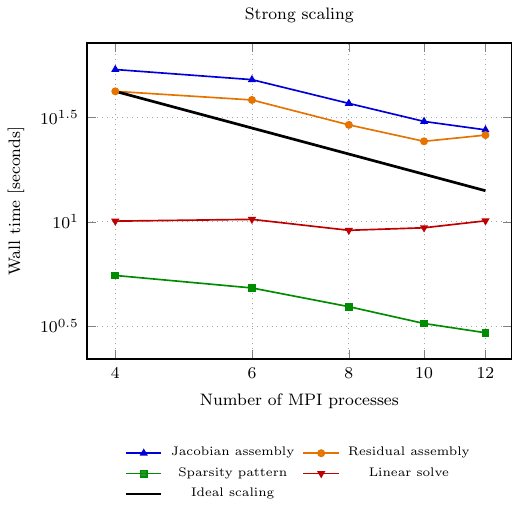}
    \caption{$50\mathrm{k}$ vessels, $500{,}010$ DoF.}
    \label{fig:50k_network_scaling}
  \end{subfigure}\hfill
  \begin{subfigure}[t]{0.48\linewidth}
    \centering
    \includegraphics[width=\linewidth]{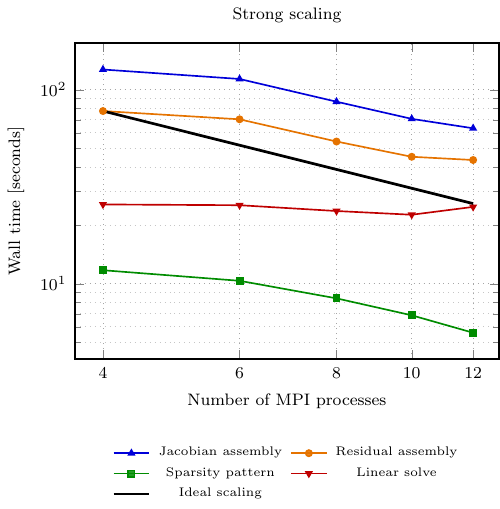}
    \caption{$100\mathrm{k}$ vessels, $1{,}000{,}010$ DoF.}
    \label{fig:100k_network_scaling}
  \end{subfigure}
  \caption{Strong-scaling results at the final time $T=1\,\mathrm{s}$. Each
  panel shows wall time for the dominant kernels versus the actual MPI process
  counts.}
  \label{fig:all_scaling}
\end{figure}

\section{Conclusion}
We have presented an HDG--HLL discretization for one-dimensional blood flow
on complex arterial networks with Riemann-coupled junctions and Windkessel
terminals. The hybrid traces place the network transmission conditions and the
terminal variables in a common global system, while the spatially discrete
problem has an explicit differential--algebraic structure: cell and capacitive
terminal variables are differential, whereas the network traces are
algebraic. This structure allows the complete flow network to be advanced through
one global residual and Jacobian structure at each time step. The analysis establishes local solvability and a
semi-discrete energy estimate under the stated positivity and subcriticality
assumptions. Manufactured-solution tests, the aortic bifurcation and the
$37$-artery benchmark verify the accuracy and physiological behavior of the method. The cerebral arteriolar experiment shows that the junction and terminal coupling remains robust on a large, deeply branching tree with thousands of resistive terminals, reproducing the mean pressure distribution of the reference study. The large-network experiments demonstrate its parallel scalability.

The DAE formulation is also the principal interface for future multiphysics
extensions. In particular, an external nonlinear elasticity model can be
introduced by augmenting the flow residual with elastic unknowns and interface
equations, thereby retaining a single monolithic Newton system. This is
consistent with geometrical multiscale models in which one-dimensional
pressure and area data drive a three-dimensional elastic problem
\cite{heltai2021multiscale,belponer2023reduced}. Future work will
address the two-way exchange between the flow tube law and the external
elastic response, as well as the corresponding fully coupled analysis and
preconditioning strategy.

\section*{Code and reproducibility}
The public \texttt{MetricFlowX} repository~\cite{metricflowx} contains the
source code, authoritative parameter files, dataset download script and
experiment scripts. Its container launcher supports Docker, Podman and
Apptainer and runs the compiled application without a local installation of
\texttt{deal.II} or the solver.

For example, the aortic-bifurcation case can be prepared and run from the
repository root as follows:
\begin{small}
\begin{verbatim}
./scripts/download_datasets.sh datasets
python3 scripts/reproduce_paper.py bifurcation \
  --runtime docker \
  --image ghcr.io/luca-heltai/metric-flow-x:main \
  --dataset-root datasets --output-root runs
\end{verbatim}
\end{small}
Replacing \texttt{docker} with \texttt{podman} or \texttt{apptainer} selects
the corresponding installed container runtime. Inputs and outputs are shared
with the host working directory; the prepared parameter file and run metadata
are written under \texttt{runs/bifurcation/}. The same experiment script
provides the remaining cases and their parameter variations. For each rerun,
the code revision, dataset checksums and container image digest should be
recorded together with the numerical output.

\section*{Acknowledgments}
RD and LH acknowledge support of the European Research Council (ERC) under the European Union's Horizon~2020 research and innovation programme (call HORIZON-EUROHPC-JU-2023-COE-03, grant agreement No.~101172493 ``dealii-X'').
LH acknowledges partial support by King Abdullah University of Science and Technology Research Funding (KRF) under Award No.~ORFS-2025-CRG13-6911.3. We thank Marco Feder (University of Pisa) and Davide Polverino (University of Pisa) for helpful discussions.

\section*{Use of AI tools in the writing process}
During the preparation of this work, the authors used ChatGPT v6.0 in order to improve
readability and language. The authors reviewed and edited the content as needed and
take full responsibility for the content of the published article.

\bibliographystyle{elsarticle-num}
\bibliography{./blood_flow}
\clearpage

\end{document}